\documentclass[reqno,a4paper,11pt]{amsart}

\usepackage[pdfpagelabels]{hyperref}
\usepackage{amssymb}
\usepackage{url}

\newtheorem{theorem}{Theorem}
\newtheorem{lemma}[theorem]{Lemma}
\newtheorem{corollary}[theorem]{Corollary}
\newtheorem{proposition}[theorem]{Proposition}

\theoremstyle{definition}

\theoremstyle{remark}
\newtheorem{remark}{Remark}
\newtheorem{example}{Example}

\newcommand{\abs}[1]{\lvert#1\rvert}
\newcommand{\nm}[1]{\lVert#1\rVert}

\newcommand{\D}{\mathbb{D}}

\newcommand{\N}{\mathbb{N}}
\newcommand{\NC}{\mathcal{N}}

\newcommand{\R}{\mathbb{R}}

\newcommand{\C}{\mathbb{C}}

\newcommand{\hp}{{\rm Har^+}}
\renewcommand{\phi}{\varphi}
\DeclareMathOperator{\Real}{Re}
\DeclareMathOperator{\Imag}{Im}

\DeclareMathOperator{\Int}{Int}
\DeclareMathOperator{\diam}{diam}

\newcommand{\BMOA}{\rm BMOA}

\renewcommand{\H}{{\rm Hol}}

\begin{document}

\title[Converse growth estimates]{Converse growth estimates for ODEs\\ with slowly growing solutions}
\thanks{The author is supported  in part by the Academy of Finland \#286877.}

\author{Janne Gr\"ohn}
\address{Department of Physics and Mathematics, University of Eastern Finland,\newline
\indent P.O.~Box 111, FI-80101 Joensuu, Finland}
\email{janne.grohn@uef.fi}

\date{\today}


\begin{abstract}
Let $f_1,f_2$ be linearly independent solutions of $f''+Af=0$, where
the coefficient $A$ is an~analytic function in the open unit disc $\D$ of $\C$.
It is shown that many properties of this differential equation can be described
in terms of the subharmonic auxiliary function $u=-\log\, (f_1/f_2)^{\#}$.
For example, the case when $\sup_{z\in\D}  |A(z)|(1-|z|^2)^2 < \infty$
and $f_1/f_2$ is normal, is characterized by the condition $\sup_{z\in\D} |\nabla u(z)|(1-|z|^2) < \infty$.
Different types of Blaschke-oscillatory equations are also described in terms of 
harmonic majorants of $u$. 

Even if $f_1,f_2$ are bounded linearly independent solutions of $f''+Af=0$,
it is possible that $\sup_{z\in\D}  |A(z)|(1-|z|^2)^2 = \infty$ or
$f_1/f_2$ is non-normal. 
These results relate to sharpness discussion of recent results in the literature,
and are succeeded by a~detailed analysis of differential equations with bounded solutions.
Analogues results for the Nevanlinna class are also considered, by taking
advantage of Nevanlinna interpolating sequences.

It is shown that, instead of considering solutions with prescribed zeros, it is possible to
construct a~bounded solution of $f''+Af=0$ in such a way that it solves
an interpolation problem natural to bounded analytic functions, while
$|A(z)|^2(1-|z|^2)^3\, dm(z)$ remains to be a~Carleson measure.
\end{abstract}

\maketitle


\section{Introduction}
 
Let $\H(\D)$ be the collection of analytic functions in the open unit disc 
$\D$ of the complex plane $\C$.  For $0\leq \alpha < \infty$,
let $L^\infty_\alpha$ denote the space of 
$f:\D \to \C$ for which $\nm{f}_{L^\infty_\alpha} = \sup_{z\in\D} |f(z)|  (1-|z|^2)^{\alpha}<\infty$,
and write $H^\infty_\alpha=L^\infty_\alpha \cap \H(\D)$ and $H^\infty = H^\infty_0$ for short.
We are interested in the relation between the growth of the coefficient $A\in\H(\D)$ and
the oscillation and growth of solutions of
\begin{equation} \label{eq:de2}
  f''+Af=0.
\end{equation}
By \cite[Theorems~3-4]{S:1955}, the following conditions are equivalent:
\begin{itemize}
\item[\rm (i)]
$A\in H^\infty_2$;

\item[\rm (ii)]
zero-sequences of all non-trivial solutions ($f\not\equiv 0$) of \eqref{eq:de2} are separated with
respect to the hyperbolic metric.
\end{itemize}
We refer to \cite{CGHR:2013} for a~far reaching generalization concerning the connection between the growth 
of the coefficient $A\in\H(\D)$
and the minimal separation of zeros of non-trivial solutions of~\eqref{eq:de2}.
It has been unclear whether 
\begin{itemize}
\item[\rm (iii)]
all solutions of \eqref{eq:de2} belong to the Korenblum space $\bigcup_{0< \alpha < \infty} H^\infty_\alpha$,
\end{itemize}
is equivalent to the conditions above. 
Recall that, if $f_1,f_2$ are linearly independent solutions of \eqref{eq:de2} for $A\in\H(\D)$,
then the Wronskian determinant $W(f_1,f_2) = f_1 f_2'-f_1'f_2$ reduces to a~non-zero complex constant,
and consequently, any solution of~\eqref{eq:de2} can be written as a~linear combination of $f_1,f_2$.

In view of results in the literature, the condition (iii) is a~natural candidate for a~description
of the growth of solutions of \eqref{eq:de2} under (i).
Pommerenke used a~classical comparison theorem \cite[Example~1]{P:1982} to prove that
 ${\rm (i)} \Rightarrow  {\rm (iii)}$. This implication has been rediscovered with different methods:
growth estimates \cite[Theorem 4.3(2)]{H:2000}, \cite[Theorem~3.1]{HKR:2007}; successive approximations \cite[Theorem~I]{G:2011};
and straight-forward integration \cite[Theorem~2]{GR:2017}, \cite[Corollary~4(a)]{HKR:2016}. 
We point out that, even if $\nm{A}_{H^\infty_2}>0$
is arbitrarily small, some solutions of \eqref{eq:de2} may be unbounded.
Any coefficient condition $A\in H^\infty_\alpha$ for $0<\alpha<2$ implies boundedness
of all solutions of \eqref{eq:de2} by \cite[Theorem 4.3(1)]{H:2000}. For more involved 
growth estimates in the case of slowly growing solutions, see~\cite{G:2018, GHR:preprint}.

The difficulty in the converse assertion ${\rm (iii)} \Rightarrow  {\rm (i)}$ lies in the fact that
the assumption concerns \emph{all} solutions. The existence of one non-trivial slowly growing solution is not sufficient,
as $f(z)=\exp(-(1+z)/(1-z))$ is a~bounded solution of~\eqref{eq:de2} for $A(z)=-4z/(1-z)^4$, $z\in\D$.
Two classical methods to attack problems of this type are the Bank-Laine approach and arguments based on the Schwarzian derivative.
In the former case, let $E=f_1f_2$ denote the product of two linearly independent solutions
of~\eqref{eq:de2} for $A\in\H(\D)$. By \cite[pp.~76--77]{L:1993}, 
\begin{equation*}
4 A = \left( \frac{E'}{E} \right)^2 - \left( \frac{W(f_1,f_2)}{E} \right)^2 - 2 \, \frac{E''}{E}.
\end{equation*}
The Bank-Laine representation is usually used in conjunction with estimates that appear 
in Wiman-Valiron and Nevanlinna theories.
The latter method is based on \cite[Theorem~6.1]{L:1993}: if $f_1,f_2$ are
linearly independent solutions of~\eqref{eq:de2} for $A\in\H(\D)$, then
$w=f_1/f_2$ is a~locally univalent meromorphic function in $\D$ such that the Schwarzian derivative
\begin{equation*}
S_w = \left( \frac{w''}{w'} \right)' - \frac{1}{2} \left( \frac{w''}{w'} \right)^2
\end{equation*}
is not only analytic in $\D$ but also
satisfies $S_w=2A$. Both approaches represent 
the coefficient function $A$ in terms of the linearly independent solutions $f_1,f_2$,
and are indispensable tools in the case of fast growing solutions (and also in oscillation theory).
However, if all solutions are slowly growing functions in $\D$, then
neither of these techniques seem to be sufficiently delicate to produce sharp growth estimates for 
the coefficient $A$.


\section{Results}

Many of the following results are \emph{converse growth estimates}
as they measure the growth of the coefficient in terms of solutions.
We begin with studying equations with bounded solutions.
The preliminary results in
Section~\ref{sec:bdd} not only set the stage for forthcoming findings but also
provide a~sharpness discussion for \cite{G:2018,S:2012}.
The significant part of this article is devoted to the study of the subharmonic auxiliary function
$u=-\log\, (f_1/f_2)^{\#}$ where $f_1,f_2$ are linearly independent solutions of~\eqref{eq:de2}.
This approach leads to several new characterizations which are, in essence, based on identities obtained in Section~\ref{sec:id}.
Our intention is to compare properties of $u$ to the coefficient $A$, to the quotient $f_1/f_2$ and
to any non-trivial solution of \eqref{eq:de2}.
Results concerning equations with bounded solutions have natural counterparts
in the setting of the Nevanlinna class, which are considered in Section~\ref{sec:nevint}.
These results depend on recent advances concerning Nevanlinna interpolating sequences.
Finally, in Section~\ref{sec:fixed}, we show that fixed points can be prescribed 
for a~solution of \eqref{eq:de2} in such a~way that all solutions remain bounded.


\subsection{Bounded solutions} \label{sec:bdd}

The following result indicates that
the implication ${\rm (iii)} \Rightarrow {\rm (i)}$, mentioned in the Introduction, fails to be true.


\begin{theorem} \label{thm:converse}
Consider the differential equation \eqref{eq:de2} in $\D$.
\begin{enumerate}
\item[\rm (i)]
There exists $A\in \H(\D) \setminus H^\infty_2$
such that all solutions of \eqref{eq:de2} are bounded.

\item[\rm (ii)]
Let $0<p<\infty$.
There exists $A\in \H(\D) \setminus H^\infty_2$
such that all solutions of~\eqref{eq:de2} belong to $H^\infty_p$ while
one of the solutions is non-normal.
\end{enumerate}
\end{theorem}

The class of normal functions consists of those meromorphic functions for which
$\sup_{z\in\D} \, w^{\#}(z) (1-|z|^2) < \infty$, where $w^\# = |w'|/(1+|w|^2)$ is the spherical derivative.
Function $w$ is normal if and only if
$\{ w \circ \phi : \text{$\phi$ conformal automorphism of $\D$}\}$ is a normal family in~$\D$ in
the sense of Montel \cite{LV:1957}. We consider
the normality of solutions of \eqref{eq:de2} as well as the normality
of the quotient of two linearly independent solutions.
If $A\in H^\infty_2$, then normal solutions of \eqref{eq:de2} are described
by \cite[Proposition~7]{GNR:2018}, and the case when the quotient is normal will
be characterized in Section~\ref{sec:pw}.
Note that the coefficient condition $A\in H^\infty_2$
allows non-normal solutions by \cite[Theorem~3]{G:2017_1} and \cite[Theorem~1]{G:2017}; and
even the normality of all solutions 
is not sufficient for $A\in H^\infty_2$ by Theorem~\ref{thm:converse}(i) above.

If $f_1, f_2 \in H^\infty$ are linearly independent solutions
of \eqref{eq:de2} for $A\in\H(\D)$, then $A\in H^\infty_3$ by Steinmetz's result \cite[p.~130]{S:2012}. 
Theorem~\ref{thm:converse}(i) shows that this result cannot be improved to $A\in H^\infty_2$.
The intermediate conclusion $A\in H^\infty_{\alpha}$ for $\alpha=5/2$ has been obtained in \cite[Theorem~6]{G:2018} under 
the weaker assumption $f_1,f_2\in \mathcal{B}$,
while the question of finding the best possible $\alpha$ remains open.
Here $\mathcal{B}$ is the Bloch space,
which contains $f\in\H(\D)$ for which $\nm{f}_{\mathcal{B}} = \nm{f'}_{H^\infty_1}<\infty$. 
The desired conclusion $A\in H^\infty_2$ has been obtained in \cite[Theorem~7]{G:2018} 
under the additional assumption $\inf_{z\in\D}( |f_1(z)| + |f_2(z)| )>0$.
We proceed to state two generalizations in this respect.
Theorem \ref{thm:converse_final} in Section~\ref{sec:ppp} shows that it is not necessary to take the
infimum over the whole unit disc while Theorem~\ref{thm:converse_finalx} below implies
that we may take the infimum of a~function which is significantly larger than $|f_1|+|f_2|$.
The latter generalization is based on having specific information about the structure of the
ideal $I_{H^\infty}(f_1,f_2)$ generated by the solutions $f_1,f_2 \in H^\infty$.

A~positive Borel measure $\mu$ on~$\D$ is called a~Carleson measure, if for fixed $0<p<\infty$
there exists $C=C(p)$ with $0<C<\infty$ such that
\begin{equation*}
\int_{\D} |f(z)|^p \, d\mu(z) \leq C \, \lim_{r\to 1^-} \frac{1}{2\pi} \int_0^{2\pi} |f(re^{i\theta})|^p \, d\theta = C \, \nm{f}_{H^p}^p,
\quad f\in\H(\D).
\end{equation*}
Here $H^p$ is the standard Hardy space.
By \cite[Lemma~3.3, p.~231]{G:2007}, such measures~$\mu$ are characterized by
$\sup_{a\in\D}  \int_\D |\varphi_a'(z)| \, d\mu(z) < \infty$,
where $\varphi_a(z)=(\zeta-z)/(1-\overline{a}z)$ is a~conformal
automorphism of $\D$ which coincides with its own inverse.
Since $|A|^2$ is subharmonic for $A\in\H(\D)$, 
we deduce $A\in H^\infty_2$ whenever $|A(z)|^2(1-|z|^2)^3\, dm(z)$ is a~Carleson measure.
This Carleson measure condition appears several times in the literature:
in connection to solutions of \eqref{eq:de2} with uniformly separated zeros \cite{G:2017, GN:2017}
and in relation to solutions in Hardy spaces~\cite{GHR:preprint, GNR:2018}.


\begin{theorem} \label{thm:converse_finalx}
If $f_1,f_2\in H^\infty$ are linearly independent solutions of \eqref{eq:de2} for $A\in\H(\D)$ such that
\begin{equation} \label{eq:tolo}
  \inf_{a\in\D} \, \sum_{k=0}^n \left( \big| (f_1 \circ \varphi_a)^{(k)}(0)\big| 
    + \big| (f_2 \circ \varphi_a)^{(k)}(0)\big| \right) > 0
\end{equation}
for some $n\in\N \cup \{0\}$, then $|A(z)|^2(1-|z|^2)^3\, dm(z)$ is a~Carleson measure.
\end{theorem}

Let $f_1,f_2\in H^\infty$ be linearly independent solutions of \eqref{eq:de2} for $A\in\H(\D)$.
In~\cite{S:2012}, Steinmetz proved $(f_1/f_2)^\#\in L^\infty_2$ and asked whether this
can be improved to $(f_1/f_2)^\#\in L^\infty_1$? It turns out that
Steinmetz's result is best possible up to a~multiplicative constant.
Recall that the sequence $\{z_n\}\subset \D$ is said to be uniformly separated, if it is separated in the hyperbolic metric
and $\sum_n (1-|z_n|)\delta_{z_n}$ is a~Carleson measure. Here $\delta_{z_n}$ is the Dirac measure with point mass at $z_n\in\D$.


\begin{theorem} \label{thm:steinmetz}
Let $\Lambda\subset\D$ be uniformly separated.
Then, there exists $A\in \H(\D)$ such that $|A(z)|^2(1-|z|^2)^3\, dm(z)$ is a~Carleson measure
and \eqref{eq:de2} admits two linearly independent solutions $f_1,f_2\in H^\infty$ such that
$\inf_{z_n\in\Lambda} \, (f_1/f_2)^\#(z_n) (1-|z_n|^2)^2 >0$.
\end{theorem}

Instead of considering prescribed zeros 
of solutions --- which is the approach in Theorem~\ref{thm:steinmetz}, among many other results --- 
we may also consider solutions which satisfy an~interpolation problem natural for bounded analytic functions.
Such result has been the objective of recent research. Our solution to this problem is based on
combining classical interpolation results by Earl and \O yma.


\begin{theorem} \label{thm:intp}
Let $\{z_n\}\subset\D$ be uniformly separated and $\{w_n\}\subset\C$ bounded.
Then, there exists $A\in \H(\D)$ such that $|A(z)|^2(1-|z|^2)^3\, dm(z)$ is a~Carleson measure,
\eqref{eq:de2} admits a~solution $f\in H^\infty$ which satisfies $f(z_n)=w_n$ for all~$n$,
while all solutions of \eqref{eq:de2} are bounded.
\end{theorem}

In Section~\ref{sec:steinmetz} we consider oscillation of solutions of
such differential equations whose solutions are bounded, and concentrate on the zeros and critical points.


\subsection{Identities} \label{sec:id}
We take a~short side-track to consider properties of the differential
equation~\eqref{eq:de2} assuming that the coefficient $A$ is merely analytic in $\D$.
Suppose for a~moment that $f$ is a~zero-free solution of \eqref{eq:de2}. 
In this case $\log f \in \H(\D)$ and 
\begin{equation} \label{eq:orig_rep}
A = - f''/f =  - ( \log f)'' - \big( (\log f)' \big)^2.
\end{equation}
Our next objective is to obtain a~similar representation which
takes account on both linearly independent solutions 
and allows them to have zeros in $\D$. Let
\begin{equation*}
\partial f = \frac{1}{2} \left( \frac{\partial f}{\partial x} - i \, \frac{\partial f}{\partial y} \right),
\quad
\overline{\partial} f = \frac{1}{2} \left( \frac{\partial f}{\partial x} + i \, \frac{\partial f}{\partial y} \right),
\end{equation*}
denote the complex partial derivatives of $f$.
Note that $\partial f$ and $\overline{\partial} f$ exist as long as $\partial f / \partial x$ and $\partial f / \partial y$
exist, and then the gradient $\nabla f = (\partial f / \partial x, \partial f / \partial y)$
satisfies $|\nabla f|^2 = 2 ( |\partial f|^2 + |\overline{\partial} f|^2 )$.
If $f$ has continuous second-order derivatives (denoted by $f\in C^2$),
then the Laplacian $\Delta f$ can be written in the form
$\Delta f = 4 \, \overline{\partial} \partial f =  4 \, \partial \overline{\partial} f$. 

We have been unable to find a~reference for the following result.


\begin{theorem} \label{thm:repres}
Let $f_1,f_2$ be linearly independent solutions of \eqref{eq:de2} for \mbox{$A\in\H(\D)$},
and define $u = -\log \, (f_1/f_2)^{\#}$. Then,
\begin{enumerate}
\item[\rm (i)]
$\Delta u = 4 \, e^{-2u}$;

\item[\rm (ii)]
$\Delta u + | \nabla u |^2 =  e^{-u} \Delta e^u$;

\item[\rm (iii)]
$A = - \partial^2 u - (\partial u)^2$.
\end{enumerate}
\end{theorem}

Let $f_1,f_2$ be linearly independent solutions of \eqref{eq:de2} for $A\in\H(\D)$.
The function $u = -\log \, (f_1/f_2)^{\#}$ has several interesting properties, which
make up the bulk of this paper. The underlying reason for the relevance of $u$
is its connection to regular conformal metrics of constant curvature.
Actually, $u$ is closely related to the general solution of Liouville's equation
in the case of $\D$. This point of view is elaborated further in Remark~\ref{remark:conformalmetric}, Section~\ref{sec:6}.
Nevertheless, we choose to proceed without the notation of conformal metrics.

Theorem~\ref{thm:repres}(i) 
implies $\Delta u = 4 \, ( (f_1/f_2)^{\#})^2 \geq 0$. Therefore $u$ is subharmonic, 
and $r \mapsto (1/(2\pi)) \int_0^{2\pi} u(re^{i\theta}) \, d\theta$ is a~non-decreasing and
convex function of $\log r$. 
Theorem~\ref{thm:repres}(iii) is a~counterpart of \eqref{eq:orig_rep}.
As $W(f_1,f_2)$ is a~non-zero complex constant,
$ \partial u  = (f_1'\overline{f}_1+f_2' \overline{f}_2)/(|f_1|^2+|f_2|^2)$
is finite-valued throughout~$\D$.


\subsection{Blaschke-oscillatory equations} \label{sec:bo}

The differential equation \eqref{eq:de2} is said
to be Blaschke-oscillatory, if $A\in\H(\D)$ and the zero-sequence $\{z_n\}$ of any non-trivial solution of \eqref{eq:de2}
satisfies the Blaschke condition $\sum_n (1-|z_n|)<\infty$. Such differential equations
are characterized by the fact that the quotient of any two linearly
independent solutions belongs to the Nevanlinna class \cite[Lemma~3]{H:2013}.
The Nevanlinna class $\NC$ consists of those meromorphic functions $w$ in~$\D$
such that $\int_{\D} w^\#(z)^2 (1-|z|^2) \, dm(z) <\infty$; see Section~\ref{sec:nevrem}. Meromorphic function $w$
is said to be of uniformly bounded characteristic, that is $w\in \rm{UBC}$, if $w^\#(z)^2 (1-|z|^2) \, dm(z)$
is a~Carleson measure. We refer to \cite[Theorem~3]{P:1991} for more details.

Let $u\not\equiv -\infty$ be a~subharmonic function in $\D$. Harmonic function $h$ is said to be a~harmonic majorant
for $u$ if $u \leq h$ in $\D$. The least harmonic majorant $\hat{u}$ is a harmonic majorant 
which is point-wise smaller than any other harmonic majorant for $u$. 
If $f\in\H(\D)$, then it is well-known that $f\in\NC$ if and only if $\log^+ |f|$ admits a~harmonic majorant, while
$f\in H^p$ if and only if $|f|^p$ has a~harmonic majorant.


\begin{theorem} \label{theorem:new2}
Let $f_1,f_2$ be linearly independent solutions of \eqref{eq:de2} for $A\in\H(\D)$,
and define $u = -\log \, (f_1/f_2)^{\#}$. Then,
\begin{enumerate}
\item[\rm (i)]
$f_1/f_2\in \NC$ if and only if~$u$ has a harmonic majorant;

\item[\rm (ii)]
$f_1/f_2\in \NC$ and is normal if and only if $u_a(z)= u(a+(1-|a|)z)-u(a)$, $a\in\D$, have harmonic majorants
with $\sup_{a\in\D} \widehat{u_a}(0) < \infty$;

\item[\rm (iii)]
$f_1/f_2 \in \rm{UBC}$ if and only if $\sup_{a\in\D} ( \hat{u}(a) - u(a) ) < \infty$.
\end{enumerate}
Moreover,
\begin{enumerate}
\item[\rm (iv)]
all solutions of \eqref{eq:de2} belong to $\NC$ if and only if~$u$ has a~positive harmonic majorant;

\item[\rm (v)]
all solutions of \eqref{eq:de2} belong to $H^p$, for $0<p<\infty$, 
if and only if $\exp(\frac{p}{2} \, u)$ has a~harmonic majorant;

\item[\rm (vi)]
all solutions of \eqref{eq:de2} belong to $H^\infty$
if and only if $\exp(u) \in L^\infty$.
\end{enumerate}
\end{theorem}

Recall that 
the following conditions are equivalent for any subharmonic function~$u$ in the unit disc
(see \cite[p.~66]{G:2007} for more details): (a) $u$ has a~positive harmonic majorant;
(b) the subharmonic function $u^+ = \max \{ u, 0\}$ has a~harmonic majorant;
(c) $u$ is majorized by a~Poisson integral of a~finite measure on $\partial\D$.
In Theorem~\ref{theorem:new2},
it is possible that $u$ admits a~harmonic majorant which takes negative values,
since there are Blaschke-oscillatory equations \eqref{eq:de2}
whose non-trivial solutions lie outside~$\NC$ \cite[Section~4.3]{H:2013}. 
Although the items (iv)--(vi) are immediate,
their assertions raise an interesting observation. 
Since we may describe the behavior of all solutions of \eqref{eq:de2} in terms of
$f_1/f_2$, no essential
information is reduced in this quotient. In Remark~\ref{remark:boder}, Section~\ref{sec:nevrem}, we illustrate that 
the growth of solutions of Blaschke-oscillatory equations
is severely restricted.

There are normal functions which do not belong to $\NC$.
Classical example of such a~function is the elliptic modular function \cite[p.~57]{LV:1957}.
If $f_1,f_2$ are linearly independent solutions of \eqref{eq:de2} for $A\in\H(\D)$, then
$f_1/f_2\in\mathcal{N}$ provided that $f_1/f_2$ is normal and
the set where $|f_1|^2+|f_2|^2$ takes small values, is not too large.


\begin{proposition} \label{prop:eset}
Let $f_1,f_2$ be linearly independent solutions of \eqref{eq:de2} for $A\in\H(\D)$.
The differential equation \eqref{eq:de2} is Blaschke-oscillatory if $f_1/f_2$
is normal and there exists $0<\delta<\infty$ such that
$\int_{\{z\in\D \; \! : \; \! |f_1(z)|^2+|f_2(z)|^2 < \delta\}} dm(z)/(1-|z|^2) < \infty$.
\end{proposition}


\subsection{Nevanlinna interpolating sequences} \label{sec:nevint}

By recent advances
concerning free interpolation in $\NC$ \cite{HMNT:2004, HMN:preprint, MNT:preprint},
there is an astounding resemblance between uniformly separated sequences and Nevanlinna interpolating sequences.
Therefore the following results can be interpreted as Nevanlinna
analogues of ones that are either presented in Section~\ref{sec:bdd} or already appear in the literature.

Sequence $\Lambda\subset \D$ is 
called (free) interpolating for $\NC$ if the trace of 
$\NC$ on $\Lambda$ is ideal \cite[p.~3]{HMNT:2004}.
That is, for any $g\in\NC$ and for any bounded sequence $\{w_n\}\in\C$, there exists $f\in\NC$
such that $f(z_n) = w_n \, g(z_n)$ for all $z_n\in\Lambda$.
The collection of (free) interpolating sequences for $\NC$ is denoted by $\Int\NC$. 
Note that $\Lambda\in \Int\NC$ if and only if the trace $\NC \mid \Lambda$ contains
all bounded sequences \cite[Remark~1.1]{HMNT:2004}, and in particular,
all sequences in $\Int\NC$ satisfy the Blaschke condition.

Let $\hp(\D)$ denote the space of positive harmonic functions in $\D$. By \cite[Theorem~1.2]{HMNT:2004},
$\Lambda\in\Int\NC$ if and only if there exists $h\in\hp(\D)$ such that
\begin{equation} \label{eq:nevsep}
  \prod_{z_k\in\Lambda\setminus \{z_n\}} \left| \frac{z_k-z_n}{1-\overline{z}_kz_n}\right|
  \geq e^{-h(z_n)}, \quad z_n\in\Lambda.
\end{equation}
The reader is invited to compare \eqref{eq:nevsep} to 
the classical description \eqref{eq:sep} of uniformly separated sequences,
which are precisely the interpolating sequences for $H^\infty$.


\begin{theorem} \label{thm:n3}
Let $\Lambda\in\Int\NC$.
Then, there exist $h \in\hp(\D)$ and $A\in \H(\D)$ such that $|A(z)|(1-|z|^2)^2 \leq e^{h(z)}$, $z\in\D$,
and \eqref{eq:de2} admits a~non-trivial solution whose zero-sequence is $\Lambda$.
\end{theorem}

By \cite[Corollary~1.9]{HMNT:2004},
Theorem~\ref{thm:n3} allows us the prescribe any separated Blaschke sequence
to be a~zero-sequence of a~non-trivial solution of \eqref{eq:de2}. Theorem~\ref{thm:n3}
should be compared to \cite[Theorem~1]{G:2017}, according to which any separated sequence
of sufficiently small upper uniform density can appear as a~subset of the zero-sequence of a~non-trivial solution
of \eqref{eq:de2} under the coefficient condition $A\in H^\infty_2$. The coefficient condition in Theorem~\ref{thm:n3}
is of different nature as it controls the growth in an~average sense. On one hand, the restriction
$|A(z)|(1-|z|^2)^2 \leq e^{h(z)}$, $z\in\D$ and $h\in\hp(\D)$, passes through
functions such as $A(z) = (e/(1-z))^k$ for any $0<k<\infty$.
On the other hand, it implies that there exists $0<C<\infty$ such that
\begin{equation} \label{eq:Aclose}
\int_0^{2\pi} \log^+ |A(re^{i\theta})| \, d\theta \leq 2 \log^+ \frac{1}{1-r} + C, \quad r\to 1^-,
\end{equation}
which is an~estimate that cannot be improved even if $A\in H^\infty_2$. 
Estimate \eqref{eq:Aclose} reveals that such coefficient $A$ lies close to $\NC$
as it is non-admissible.

The following result is an~analogue of \cite[Theorem~5]{G:2017_1}, and
is related to the classical $0,1$-interpolation 
result due to Carleson \cite[Theorem~2]{C:1962}. The Nevanlinna
counterpart of Carleson's result is presented in Section~\ref{sec:carleson_a}.


\begin{theorem} \label{thm:carleson}
Assume that $\alpha,\beta\in\C\setminus\{0\}$ are distinct values.
Let $\{z_n\},\{\zeta_n\}$ be any Blaschke sequences, and let $B_{\{z_n\}}$ and $B_{\{\zeta_n\}}$ be
the corresponding Blaschke products. If there exists $h\in\hp(\D)$ such that
\begin{equation} \label{eq:cca}
\big|B_{\{z_n\}}(z)\big| + \big|B_{\{\zeta_n\}}(z)\big| \geq e^{-h(z)}, \quad z\in\D,
\end{equation} 
then there exists
$A \in \H(\D)$ and $H\in\hp(\D)$ such that
$|A(z)|(1-|z|^2)^2 \leq e^{H(z)}$, $z\in\D$, and \eqref{eq:de2} admits a~solution
$f$ with $f(z_n)=\alpha$ and $f(\zeta_n)=\beta$ for all $n$.
\end{theorem}

We turn to study differential equations with solutions in $\NC$. 
It turns out that Steinmetz's approach from \cite[Theorem, p.~129]{S:2012} applies with obvious changes.


\begin{theorem} \label{thm:n1}
If $f_1,f_2\in \mathcal{N}$ are linearly independent solutions of \eqref{eq:de2} for $A\in\H(\D)$, then 
there exists $H\in\hp(\D)$ such that $|A(z)|(1-|z|^2)^3 \leq e^{H(z)}$ and $(f_1/f_2)^{\#}(z)(1-|z|^2)^2 \leq e^{H(z)}$, $z\in\D$.
\end{theorem}

We may also ask when the stronger estimate
$|A(z)|(1-|z|^2)^2 \leq e^{H(z)}$, $z\in\D$, is obtained? The following result is
analogous to Theorem~\ref{thm:converse_finalx}; generalization of the assumption \eqref{eq:toloN}
to higher derivatives is left to the interested reader.


\begin{theorem} \label{thm:n2}
If $f_1,f_2\in \mathcal{N}$ are linearly independent solutions of \eqref{eq:de2} for $A\in\H(\D)$ such that
\begin{equation} \label{eq:toloN}
\sum_{j=1,2} \Big( |f_j(z)| + |f_j'(z)|(1-|z|^2) \Big) \geq e^{-h(z)}, \quad z\in\D,
\end{equation}
for $h\in\hp(\D)$, then there exists $H\in\hp(\D)$ such that $|A(z)|(1-|z|^2)^2 \leq e^{H(z)}$, $z\in\D$.
\end{theorem}

The sequence $\Lambda\subset\D$ is called $h$-separated, if 
there exists $h\in\hp(\D)$ such that the pseudo-hyperbolic discs $\Delta_p(z_n,e^{-h(z_n)})$, $z_n\in\Lambda$,
are pairwise disjoint. 
Recall that the pseudo-hyperbolic disc of radius $0<\delta<1$, centered
at $z\in\D$, is given by $\Delta_p(z,\delta) = \{ w \in\D : \varrho_p(z,w) < \delta \}$
where $\varrho_p(z,w) = |w-z|/|1-\overline{w}z|$ is the pseudo-hyperbolic distance between $z,w\in\D$.
The following result 
corresponds to Schwarz's findings \cite[Theorems~3-4]{S:1955} in the case $A\in H^\infty_2$.


\begin{proposition} \label{prop:n4}
Suppose that there exist $A\in \H(\D)$ and $H\in\hp(\D)$ such that $|A(z)|(1-|z|^2)^2 \leq e^{H(z)}$, $z\in\D$.
Then, there exists $h\in\hp(\D)$ such that 
the zero-sequence of any non-trivial solution of \eqref{eq:de2} is $h$-separated.

Conversely, suppose that $A\in \H(\D)$ and there exists $h\in\hp(\D)$ such that
the zero-sequence of any non-trivial solution of \eqref{eq:de2} is $h$-separated. 
Then, there exists $H\in\hp(\D)$ such that $|A(z)|(1-|z|^2)^2 \leq e^{H(z)}$, $z\in\D$.
\end{proposition}


\subsection{Point-wise growth restrictions} \label{sec:pw}

Function $\omega: \D \to (0,\infty)$ is said to be a~weight if it is
bounded and continuous. The weight $\omega$ is radial
if $\omega(z) = \omega(|z|)$ for all $z\in\D$, and is called regular if
it is radial and for each $0\leq s <1$ there exists a~constant $C=C(s,\omega)$ with $1\leq C<\infty$
such that 
\begin{equation} \label{eq:regular}
C^{-1} \, \omega(t)\leq \omega(r) \leq C \, \omega(t), \quad 0\leq r \leq t \leq r+s(1-r) < 1.
\end{equation}
For a~general reference for regular weights, see \cite[Chapter~1]{PR:2014}.
For a~weight $\omega$, let $L^\infty_\omega$ denote the growth space which consists of functions
$f:\D \to \C$ for which $\nm{f}_{L^\infty_\omega} = \sup_{z\in\D} |f(z)| \, \omega(z) <\infty$,
and denote $H^\infty_\omega=L^\infty_\omega\cap \H(\D)$.


\begin{theorem} \label{theorem:c2}
Let $f_1,f_2$ be linearly independent solutions of \eqref{eq:de2} for $A\in\H(\D)$,
and define $u = -\log \, (f_1/f_2)^{\#}$.
Suppose that $\omega$ is a~regular weight which satisfies $\sup_{z\in\D}\, \omega(z)/(1-|z|) < \infty$.
Then, $|\nabla u | \in L^\infty_\omega$
if and only if $A\in H^\infty_{\omega^2}$ and \mbox{$(f_1/f_2)^\#\in L^\infty_\omega$}.
\end{theorem}

The following result follows directly from Theorem~\ref{theorem:c2} 
with $\omega(z)=1-|z|^2$, $z\in\D$. This corollary concerns those
differential equations \eqref{eq:de2} which have both desired properties mentioned in Section~\ref{sec:bdd}:
$A\in H^\infty_2$ and $(f_1/f_2)^\# \in L^\infty_1$.


\begin{corollary} \label{cor:partialconverse}
Let $f_1,f_2$ be linearly independent solutions of \eqref{eq:de2} for $A\in\H(\D)$,
and define $u = -\log \, (f_1/f_2)^{\#}$.
Then, $|\nabla u|\in L^\infty_1$ if and only if 
$A\in H^\infty_2$ and $f_1/f_2$ is normal.
\end{corollary}

Corollary~\ref{cor:partialconverse} can also be deduced by combining
several results in the literature.
The first part follows from 
\cite[Theorem~6]{AM:2011},
while the second part can be concluded
from \cite[Theorem~1]{Y:2002} and \cite[Corollary to Theorem~2]{Y:2002}.
Note that $f_1/f_2$ is uniformly locally univalent provided that $A\in H^\infty_2$, which
can be seen by applying Nehari's univalency criterion \cite[Theorem~I]{N:1949} locally.


\begin{corollary} \label{cor:c2}
Let $f_1,f_2$ be linearly independent solutions of \eqref{eq:de2} for $A\in H^\infty_{\omega^2}$,
and define $u = -\log \, (f_1/f_2)^{\#}$.
Suppose that $\omega$ is a~regular weight which satisfies $\sup_{z\in\D}\, \omega(z)/(1-|z|) < \infty$.
Then, the following statements are equivalent:
\begin{enumerate}
\item[\rm (i)]
$|\nabla u |\in L^\infty_\omega$;

\item[\rm (ii)]
$(f_1/f_2)^\#\in L^\infty_\omega$;

\item[\rm (iii)]
$(|f_1'|+|f_2'|)/(|f_1|+|f_2|)\in L^\infty_\omega$;

\item[\rm (iv)]
$\Delta u \in L^\infty_{\omega^2}$.
\end{enumerate}
\end{corollary}

If $\omega(z)=1-|z|^2$, $z\in\D$, then Corollary~\ref{cor:c2}
provides a~complete description of those differential equations~\eqref{eq:de2} for $A\in H^\infty_2$,
where the quotient of two linearly independent solutions is normal.
Such characterizations are important in oscillation theory. Since normal
functions are Lipschitz-continuous, 
as mappings from~$\D$ equipped with the hyperbolic metric
to the Riemann sphere equipped with the chordal metric, 
the normality of $f_1/f_2$ implies that its
the zeros and poles (which correspond to the zeros of $f_1$ and $f_2$, respectively)
are separated in the hyperbolic metric.
Finally, we point out that
Corollary~\ref{cor:c2}(iii) does not extend 
to higher derivatives, since there are differential equations \eqref{eq:de2} with $A\in\H(\D)$ and
$|A| = (|f_1''|+|f_2''|)/(|f_1|+|f_2|)\in L^\infty_2$
such that the quotient $f_1/f_2$ of linearly independent solutions $f_1,f_2$ is non-normal;
see \cite{L:1973} and Theorem~\ref{thm:steinmetz}.


\subsection{Prescribed fixed points} \label{sec:fixed}

The point $z_0\in\D$ is said to be a fixed point of $f\in\H(\D)$ if $f(z_0)=z_0$.
There are a~lot of known results according to which
zeros and critical points (i.e., zeros of the derivative) can be prescribed for solutions 
of~\eqref{eq:de2} for $A\in\H(\D)$. See \cite{G:2017, GH:2012, H:2011, H:2013} among many others. For example, the proof of
Theorem~\ref{thm:steinmetz} depends on such an~argument.
It turns out that fixed points can be prescribed for a~solution of~\eqref{eq:de2} under
the coefficient condition $A\in\H(\D)$
in such a way that all solutions of the differential equation remain bounded. Such differential equations were studied in detail 
in Section~\ref{sec:bdd}.


\begin{theorem} \label{thm:fixedpoints}
Let $\Lambda\subset\D$ be a~Blaschke sequence, and let $0<\varepsilon<1$.
Then, there exists a~coefficient $A\in \H(\D)$ such that
$|A(z)|^2(1-|z|^2)^3\, dm(z)$ is a~Carleson measure;
the differential equation \eqref{eq:de2} admits a~solution $f$, which satisfies $\nm{f}_{H^\infty}< 1 + \varepsilon$
and has fixed points $\{0\} \cup \Lambda$; all solutions of \eqref{eq:de2} are bounded.
\end{theorem}

If we assume that prescribed fixed points are uniformly separated, then we can go further and
dictate their type. In this paper, we make distinction between three different types:
the fixed point $z_0\in\D$ of $f\in \H(\D)$ is said to be attractive if $|f'(z_0)|<1$, 
neutral if $|f'(z_0)|=1$, and repulsive if $|f'(z_0)|>1$.


\begin{theorem} \label{thm:fixedpoints2}
Let $\Lambda\subset\D\setminus\{0 \}$ be uniformly separated.
Then, there exists a~coefficient $A\in \H(\D)$ such that
$|A(z)|^2(1-|z|^2)^3\, dm(z)$ is a~Carleson measure;
the differential equation \eqref{eq:de2} admits a bounded solution for which every point in $\Lambda$
is a~fixed point of prescribed type; all solutions of \eqref{eq:de2} are bounded.
\end{theorem}

Theorem~\ref{thm:fixedpoints2} has a~natural counterpart in the setting of Nevanlinna
interpolating sequences. Note that Theorem~\ref{thm:fixedpoints} is valid for
sequences $\Lambda\in\Int\NC$ as it is.


\begin{theorem} \label{thm:fixedpoints2N}
Let $\Lambda\subset\D\setminus\{0 \}$ and $\Lambda \in \Int\NC$.
Then, there exists a~coefficient $A\in \H(\D)$ and $H\in\hp(\D)$ such that
$|A(z)|^2(1-|z|^2)^2\leq e^{H(z)}$, $z\in\D$, and~\eqref{eq:de2} admits a~solution for which every point in $\Lambda$
is a~fixed point of prescribed type.
\end{theorem}


\section{Proof of Theorem~\ref{thm:converse}}

The following argument is based on concrete construction.


\begin{proof}[Proof of Theorem~\ref{thm:converse}]
(i) Let $0<p<1/2$, and
\begin{equation*}
f_1(z)=\exp\bigg( i \cdot \frac{p}{2\pi} \bigg( \log \frac{2i}{1-z} \bigg)^2 \, \bigg), \quad z\in\D.
\end{equation*}
Note that the function $z\mapsto 2i/(1-z)$ maps $\D$ onto $\{z\in\C: \Imag z > 1\}$.
Since
\begin{equation*}
\Real\bigg(  i \cdot \frac{p}{2\pi} \bigg( \log \frac{2i}{1-z} \bigg)^2 \, \bigg)
   = - \frac{p}{\pi} \, \log\frac{2}{|1-z|} \, \arg \frac{2i}{1-z}, \quad z\in\D,
\end{equation*}
we deduce $2^{-p} (1-|z|)^p \leq |f_1(z)| \leq 1$ for $z\in\D$.
Since $f_1$ is zero-free, we conclude $A=-f_1''/f_1 \in \H(\D)$. Moreover, $A\not\in H^\infty_2$ because
\begin{equation*}
A(z) = p \, \frac{p\big( \log\frac{2i}{1-z} \big)^2 - i \pi  \log\frac{2i}{1-z}-i\pi}{\pi^2 (1-z)^2}, \quad z\in\D.
\end{equation*}

It remains to show that \emph{all} solutions of \eqref{eq:de2} are bounded.
Note that 
\begin{equation} \label{eq:second}
f_2(z) = f_1(z) \, \int_0^z \frac{1}{f_1(\zeta)^2} \, d\zeta, \quad z\in\D,
\end{equation}
is a~bounded solution of \eqref{eq:de2}, and $f_2$ is linearly independent to $f_1$.
Here we integrate along the straight line segment.
This completes the proof of (i), since every solution of \eqref{eq:de2} is
a~linear combination of $f_1,f_2$.

(ii) Let $0<p<1/2$, and 
\begin{equation*}
f_1(z)=\exp\bigg( i \cdot \frac{p}{\pi} \bigg( \log \frac{1+z}{1-z} \bigg)^2 \, \bigg), \quad z\in\D.
\end{equation*}
Similar function has been utilized in \cite[pp.~142--143]{L:1978}. 
We point out that $f_1$ has asymptotic
values $0$ and $\infty$ at $z=1$, and hence $f_1$ is not normal. This fact alone
implies that the zero-free function $f_1$ cannot be a~solution of \eqref{eq:de2}
for $A\in H^\infty_2$; see \cite[Proposition~7]{GNR:2018}. As in the part (i), we deduce
\begin{equation} \label{eq:bounds}
\left( \frac{1-|z|}{1+|z|} \right)^p \leq |f_1(z)| \leq \left( \frac{1+|z|}{1-|z|} \right)^p, \quad z\in\D.
\end{equation}
Since $f_1$ is zero-free, we conclude $A=-f_1''/f_1 \in \H(\D)$. Moreover, $A\not\in H^\infty_2$ as
\begin{equation*}
A(z) = 8p \, \frac{2p\big( \log\frac{1+z}{1-z} \big)^2 - i \pi z \log\frac{1+z}{1-z}-i\pi}{\pi^2 (1-z^2)^2}, \quad z\in\D.
\end{equation*}

It remains to show that \emph{all} solutions of \eqref{eq:de2} belong to $H^\infty_p$.
On one hand, it is clear that $f_1\in H^\infty_p$ by \eqref{eq:bounds}. On the other hand,
\eqref{eq:second}
is a solution of \eqref{eq:de2} which is linearly independent to $f_1$. 
Since $z \mapsto \int_0^z d\zeta/f_1(\zeta)^2$ 
is bounded in $\D$, we have $f_2\in H^\infty_p$. This completes the proof of Theorem~\ref{thm:converse}.
\end{proof}


\section{Proof of Theorem~\ref{thm:converse_finalx}} \label{sec:ppp}

We offer two different proofs for Theorem~\ref{thm:converse_finalx}. 
We begin by considering a~more general result which implies Theorem~\ref{thm:converse_finalx}
as a corollary.
The following lemma indicates that any analytic function, which satisfies $H^\infty_\alpha$-type estimate
outside a~small exceptional set, actually belongs to $H^\infty_\alpha$.


\begin{lemma} \label{lemma:exc_set}
Let $f\in\H(\D)$ and $0\leq \alpha<\infty$. Then $f\in H^\infty_\alpha$,
if  there exist pairwise disjoint discs
$\Delta_p(z_n,\delta)$, $z_n\in\D$ and $0<\delta<1$, such that
\begin{equation} \label{eq:assump_est}
\sup \bigg\{ |f(z)| (1-|z|^2)^\alpha : z \in \D \setminus \bigcup_n \, \Delta_p(z_n,\delta) \bigg\} < \infty.
\end{equation}
\end{lemma}


\begin{proof}
Let $z\in \Delta_p(z_n,\delta)$ for some $n$,
and let $S$ be the supremum in~\eqref{eq:assump_est}.
By the maximum modulus principle, there exists
$\zeta\in \partial\Delta_p(z_n,\delta)$ such that 
$|f(\zeta)| = \max\big\{ |f(\xi)| : \xi\in\overline{\Delta_p(z_n,\delta)}\big\}$. By the standard estimates,
there exists a~constant $C=C(\delta)$ with $0<C<\infty$ such that 
\begin{equation*}
  |f(z)| (1-|z|^2)^\alpha 
  \leq |f(\zeta)| (1-|z|^2)^\alpha 
  \leq C^\alpha |f(\zeta)| (1-|\zeta|^2)^\alpha 
  \leq C^\alpha S.
\end{equation*}
The assertion $f\in H^\infty_\alpha$ follows.
\end{proof}

Recall that the space $\BMOA$ consists of those functions in $H^2$
whose boundary values have bounded mean oscillation on $\partial\D$,
or equivalently, of those functions $f\in\H(\D)$
for which $|f'(z)|^2(1-|z|^2)\, dm(z)$ is a~Carleson measure. We write
$\nm{f}_{\rm BMOA}^2 = \sup_{a\in\D} \, \nm{f_a}_{H^2}^2$ where $f_a(z)=f(\varphi_a(z)) - f(a)$ for $a,z\in\D$.


\begin{theorem} \label{thm:converse_final}
If $f_1,f_2\in \mathcal{B}$ are linearly independent solutions of \eqref{eq:de2} for $A\in\H(\D)$,
and there exist pairwise disjoint discs 
$\Delta_p(z_n,\delta)$, $z_n\in\D$ and $0<\delta<1$, with
\begin{equation} \label{eq:exc_ind}
\inf\bigg\{ |f_1(z)| + |f_2(z)| : z \in \D \setminus \bigcup_n \, \Delta_p(z_n,\delta) \bigg\} > 0,
\end{equation}
then $A\in H^\infty_2$. If $f_1,f_2\in\BMOA$ and the sequence $\{z_n\}\subset\D$ in \eqref{eq:exc_ind} is uniformly separated,
then $|A(z)|^2(1-|z|^2)^3\, dm(z)$ is a~Carleson measure.
\end{theorem}

The first part of Theorem~\ref{thm:converse_final} improves \cite[Theorem~7]{G:2018} 
by Example~\ref{ex:cases}(ii) below. When comparing Theorem~\ref{thm:converse_final}
to Theorem~\ref{thm:converse_finalx} note that in the former result it is not required that $f_1,f_2\in H^\infty$.


\begin{proof}[Proof of Theorem~\ref{thm:converse_final}]
Let $f_1,f_2\in \mathcal{B}$ be linearly independent solutions of \eqref{eq:de2} and suppose that
\eqref{eq:exc_ind} holds. Denote $\Omega = \bigcup_{n} \Delta_p(z_n,\delta)$.
Since
\begin{equation} \label{eq:preli}
|A| = \frac{|f_1|+|f_2|}{|f_1|+|f_2|} \, |A| = \frac{|f_1''|+|f_2''|}{|f_1|+|f_2|},
\end{equation}
we deduce
\begin{equation*}
  \sup_{z\in\D\setminus\Omega} |A(z)| (1-|z|^2)^2 \leq \frac{\nm{f_1''}_{H^\infty_2} + \nm{f_2''}_{H^\infty_2}}
  {\inf_{z\in\D\setminus\Omega} \big(|f_1(z)| + |f_2(z)|\big)}.
\end{equation*}
Since $A\in\H(\D)$, we conclude $A\in H^\infty_2$ by Lemma~\ref{lemma:exc_set}. This completes the
proof of the first part of Theorem~\ref{thm:converse_final}.

If $f_1,f_2\in \BMOA$ and $\{z_n \}\subset \D$ in \eqref{eq:exc_ind} is uniformly separated, then we write
\begin{equation*}
\sup_{a\in\D} \int_{\D} |A(z)|^2(1-|z|^2)^3 \, \frac{1-\abs{a}^2}{|1-\overline{a}z|^2} \, dm(z) = I_1 + I_2,
\end{equation*}
where $I_1,I_2$ are defined as below.
By \eqref{eq:preli} and \cite[Theorem~4.2.1]{R:2001}, we deduce
\begin{align}
I_1 & = \, \sup_{a\in\D} \int_{\D\setminus \Omega} |A(z)|^2(1-|z|^2)^3 \, \frac{1-\abs{a}^2}{|1-\overline{a}z|^2} \, dm(z) \notag\\
    & \, \lesssim  \sup_{a\in\D} \int_{\D} \big( |f_1''(z)|^2 + |f_2''(z)|^2 \big) (1-|z|^2)^3 \, \frac{1-\abs{a}^2}{|1-\overline{a}z|^2} \, dm(z)
< \infty. \label{eq:ppop}
\end{align}
Actually, \eqref{eq:ppop} is bounded above by a~constant multiple of $\nm{f_1}_{\rm BMOA}^2 + \nm{f_2}_{\rm BMOA}^2$.
Since $A\in H^\infty_2$ by the first part of the proof, standard estimates yield
\begin{align}
I_2 & = \sup_{a\in\D} \, \sum_n \, \int_{\Delta_p(z_n,\delta)} |A(z)|^2(1-|z|^2)^3 \, \frac{1-\abs{a}^2}{|1-\overline{a}z|^2} \, dm(z) \notag\\
    & \lesssim \nm{A}^2_{H^\infty_2} \, \sup_{a\in\D} \, \sum_n \frac{(1-\abs{a}^2)(1-|z_n|^2)}{|1-\overline{a}z_n|^2} < \infty. \label{eq:sum}
\end{align}
The sum in \eqref{eq:sum} is finite by the uniform separation of $\{z_n\}$. 
This completes the proof of Theorem~\ref{thm:converse_final}.
\end{proof}

If $\{z_n\}\subset \D$ is a~Blaschke sequence, then the Blaschke product
\begin{equation*}
B(z) = B_{\{z_n\}}(z) = \prod_n \frac{|z_n|}{z_n} \, \frac{z_n-z}{1-\overline{z}_n z}, \quad z\in\D,
\end{equation*}
is a bounded analytic function which vanishes precisely on $\{z_n\}$. 
Let $f_1,f_2\in H^\infty$. By \cite[Theorem~3]{T:1988}, the ideal 
\begin{equation*}
J_{H^\infty}(f_1,f_2) = \Big\{ f\in H^\infty : \text{ $\exists \, c = c(f)>0$ such that  $|f| \leq c \big(|f_1| + |f_2|\big)$} \!\Big\}
\end{equation*}
contains a~Blaschke product whose zeros form a~finite union of uniformly separated sequences if and only if \eqref{eq:tolo} holds. 
If $B$ is such a~Blaschke product, then there exists a~constant $0<\delta<1$ and a~subsequence $\{z_n'\}$
of zeros of $B$ such that the discs $\Delta_p(z_n',\delta)$, $n\in\N$, are pairwise disjoint and
\begin{equation} \label{eq:eB} 
\inf\bigg\{ |B(z)| : z \in \D \setminus \bigcup_n \, \Delta_p(z_n',\delta) \bigg\} > 0.
\end{equation}
This follows from \cite[Lemmas~1 and 3]{K-L:1969}; see also \cite[Lemma~1]{N:1994}.
Therefore Theorem~\ref{thm:converse_final} gives an immediate
proof for Theorem~\ref{thm:converse_finalx}. We also present another proof 
which, in addition, provides a~concrete representation for the coefficient $A$.


\begin{proof}[Proof of Theorem~\ref{thm:converse_finalx}]
By \eqref{eq:tolo} and \cite[Theorem~3]{T:1988}, also the ideal $I_{H^\infty}(f_1,f_2)$ contains a~Blaschke product $B$
whose zeros form a~finite union of uniformly separated sequences.
This is equivalent to the fact that there exist functions $g_1,g_2\in H^\infty$ such that $f_1 g_1 + f_2 g_2 = B$. 
Differentiate this identity twice, and then apply \eqref{eq:de2} to $f_1''$ and $f_2''$,
to obtain
\begin{equation} \label{eq:diff_twice}
A = \frac{2(f_1'g_1'+f_2'g_2') + f_1g_1''+ f_2g_2''-B''}{B}.
\end{equation}
As in the proof of Theorem~\ref{thm:converse_final}, by taking account on \eqref{eq:eB},
we conclude that $|A(z)|^2(1-|z|^2)^3\, dm(z)$ is a~Carleson measure.
\end{proof}

One of the objectives in Section~\ref{sec:bdd} was to generalize 
a~result according to which $A\in H^\infty_2$ if
$f_1,f_2\in\mathcal{B}$ are linearly independent solutions of \eqref{eq:de2} for $A\in\H(\D)$ such that 
$\inf_{z\in\D} (|f_1(z)| + |f_2(z)|) >0$. 
The Cauchy-Schwarz inequality gives
\begin{equation} \label{eq:csaa}
|W(f_1,f_2)|^2 \leq ( |f_1|^2+|f_2|^2 ) \, ( |f_1'|^2+|f_2'|^2 ).
\end{equation}
Since $f_1',f_2'\in H^\infty_1$, we deduce $|f_1(z)|+|f_2(z)| \gtrsim 1-|z|^2$, $z\in\D$,  without using
any additional assumptions.


\begin{example} \label{ex:cases}
Let $f_1,f_2\in H^\infty$ be linearly independent solutions of \eqref{eq:de2} for $A\in H^\infty_2$. 
This example concerns different situations that may happen.
\begin{enumerate}
\item[\rm (i)]
There are a~lot of examples in which $\inf_{z\in\D} ( |f_1(z)| + |f_2(z)| ) >0$.
See the discussion after the proof of \cite[Theorem~2]{G:2017}, for example.

\item[\rm (ii)]
The proof of Theorem~\ref{thm:steinmetz} below produces an~example, where
the condition \eqref{eq:exc_ind} holds; 
take $\{z_n\}$ as in Theorem~\ref{thm:steinmetz} and note that $|f_1|+|f_2|\geq |f_1|$, where
$f_1$ has the desired property.
At the same time, $|f_1(z_n)|+|f_2(z_n)| \asymp 1-|z_n|^2$
as $n\to\infty$. Not only $\inf_{z\in\D} ( |f_1(z)| + |f_2(z)| ) >0$ fails
to be true but also it breaks down in the~worst possible way.

\item[\rm (iii)]
Let $f_1(z)=(1-z^2)^{1/2}$ and $f_2(z)=(1-z^2)^{1/2} \log( (1+z)/(1-z) )$, $z\in\D$.
These functions are linearly independent solutions of \eqref{eq:de2} for the coefficient
$A(z)=1/(1-z^2)^2$, $z\in\D$, which evidently satisfies $A\in H^\infty_2$. Since both solutions have radial limit zero along the
positive real axis, the condition \eqref{eq:exc_ind} cannot hold for any pairwise disjoint
pseudo-hyperbolic discs.
\end{enumerate}
\end{example}


\section{Proofs of Theorems~\ref{thm:steinmetz} and~\ref{thm:intp}} \label{sec:steinmetz}

The first part of the proof of Theorem~\ref{thm:steinmetz}
follows directly from that of \cite[Corollary~3]{G:2017}. 
The new contribution lies in the fact that the differential equation in question admits only 
bounded solutions.


\begin{proof}[Proof of Theorem~\ref{thm:steinmetz}]
Let $B=B_\Lambda$ be the Blaschke product corresponding to the uniformly separated sequence $\Lambda=\{z_n\}$.
By \eqref{eq:sep} and Cauchy's integral formula,
$\sup_{z_n\in\Lambda} |B''(z_n)|/|B'(z_n)|^2 < \infty$.
Let $f_1=Be^{Bk}$, where $k\in H^\infty$ is a~solution of the interpolation problem
\begin{equation*}
k(z_n) = - \frac{B''(z_n)}{2 \, B'(z_n)^2}, \quad z_n\in\Lambda.
\end{equation*}
As in the proof of Theorem~\ref{thm:converse_final}, the coefficient $A=-f_1''/f_1 \in\H(\D)$
induces a~Carleson measure $|A(z)|^2(1-|z|^2)^3\, dm(z)$.
Now, $f_1$ is a~solution of \eqref{eq:de2} which has precisely the prescribed zeros $\Lambda$.

Since $\Lambda$ is uniformly separated, there exists a~constant
$0<\delta<1$ such that $\Omega = \bigcup_{z_n\in\Lambda} \Delta_p(z_n,\delta)$ 
is a~union of pairwise disjoint pseudo-hyperbolic discs. Fix any 
$\alpha\in \D \setminus \Omega$, and 
define the meromorphic function $f_2$ by
\begin{equation} \label{eq:solrep}
f_2(z) = f_1(z)\, \int_\alpha^z \frac{1}{f_1(\zeta)^2} \, d\zeta, \quad z\in\D.
\end{equation}
Choose the path of integration 
by the following rules. If $z\in \D \setminus \Omega$, then the whole path lies in $\D \setminus \Omega$.
If $z\in \Delta_p(z_n,\delta)$ for some $z_n\in\Lambda$, then the path stays in 
$(\D \setminus \Omega) \cup \Delta_p(z_n,\delta)$. 
Then, each point $z\in\D$ can be reached
by a~path which satisfies these properties and is also of uniformly bounded Euclidean length.
The following argument is standard.
In a~sufficiently small pseudo-hyperbolic neighborhood of~$\alpha$, $f_2$ represents
an~analytic function such that $f_1f_2'-f_1'f_2$ is identically one. 
As a~solution of~\eqref{eq:de2} function $f_2$ admits an~analytic continuation to~$\D$,
and this continuation agrees with the representation \eqref{eq:solrep}.

There exists a~constant
$\mu=\mu(\Lambda)$ such that $|B(\zeta)|\geq \mu>0$ for $\zeta\in\D\setminus \Omega$;
see \cite[Theorem~1]{CC:1968} for example.
We deduce
\begin{equation*}
  |f_2(z)| 
  \leq |B(z)| e^{|B(z)| |k(z)|} \int_\alpha^z \frac{|d\zeta|}{|B(\zeta)|^2 e^{-2 |B(\zeta)| |k(\zeta)|}}
  \leq \frac{e^{3 \:\! \nm{k}_{H^\infty}}}{\mu^2} \int_{\alpha}^z |d\zeta|, \quad z\in \D \setminus \Omega.
\end{equation*}
Lemma~\ref{lemma:exc_set} implies that $f_2\in H^\infty$. 
Since $W(f_1,f_2)=1$, we obtain 
\begin{align*}
 (f_1/f_2)^\#(z_n) \, (1-|z_n|^2)^2 
  & =  \frac{1}{|f_2(z_n)|^2} \, (1-|z_n|^2)^2
    = |f_1'(z_n)|^2  (1-|z_n|^2)^2\\
  & =  |B'(z_n)|^2  (1-|z_n|^2)^2, \quad z_n\in\Lambda.
\end{align*}
This completes the proof as $\Lambda$ is uniformly separated.
\end{proof}

The proof of Theorem~\ref{thm:intp} depends on a~supporting result,
which is considered next. 
Suppose that $f\in\H(\D)$, $f\colon\D \to\D$, $f(0)=0$ and $\abs{f'(0)}\geq \delta$
for some $0<\delta\leq 1$. By Cauchy's integral formula and Schwarz's lemma,
\begin{equation*}
\abs{f'(0)}-\abs{f'(z)} \leq \frac{12\, |z|}{(1-|z|)^2}, \quad z\in\D.
\end{equation*}
If $0<\eta<1$ satisfies $12\eta/(1-\eta)^2< \delta/2$, then
then $\abs{f'(z)}\geq \delta/2$ for all $\abs{z}< \eta$. 
The following lemma is a~conformally invariant version of this property.


\begin{lemma} \label{lemma:a2}
Suppose that $f\in\H(\D)$ and $f\colon\D \to\D$. Assume that there exists a~sequence $\Lambda\subset \D$
such that $\inf_{z_n\in\Lambda} \, \abs{f'(z_n)} (1-\abs{z_n}^2) \geq \delta>0$.
If $0<\eta<1$ satisfies $12\eta/(1-\eta)^2< \delta/2$,
then there exist a~constant $\nu=\nu(\delta)$ such that
\begin{equation*}
\abs{f'(z)}  (1-\abs{z}^2) \geq \nu >0, \quad z\in \bigcup_{z_n\in\Lambda} \Delta_p(z_n,\eta).
\end{equation*}
\end{lemma}


\begin{proof}
Let $z_n\in\Lambda$ be fixed, and define $g_{z_n} = \varphi_{f(z_n)} \circ f \circ \varphi_{z_n}$.
Now $g_{z_n}\colon \D\to\D$ is analytic, $g_{z_n}(0) = 0$, and
\begin{equation*}
  \abs{g_{z_n}'(0)} 
  = \big|\varphi_{f(z_n)}' \big( f( z_n) \big) \big| \, \abs{ f'(z_n)} (1-\abs{z_n}^2 )
  \geq \abs{ f'(z_n)}  (1-\abs{z_n}^2 ) \geq \delta.
\end{equation*}
The property above implies
\begin{equation*}
  \abs{g_{z_n}'(z)} 
  = \big|\varphi_{f(z_n)}' \big( f( \varphi_{z_n}(z)) \big) \big| \cdot \abs{ f'( \varphi_{z_n}(z))} \cdot \abs{ \varphi_{z_n}'(z) }
  \geq \delta/2, \quad \abs{z}< \eta.
\end{equation*}
If we denote $w=\varphi_{z_n}(z)$, then $\abs{z}<\eta$ if and only if $w\in\Delta_p(z_n,\eta)$. Consequently,
\begin{equation*}
  \abs{ f'(w)} (1-\abs{w}^2) \geq \frac{\delta}{2} \cdot \frac{1-\abs{\varphi_{z_n}(w)}^2}{1-\abs{\varphi_{f(z_n)}\big(f(w)\big)}^2}
  \big( 1-\abs{f(w)}^2 \big), \quad w\in\Delta(z_n,\eta).
\end{equation*}
Since $\varrho_p(w, z_n) < \eta$, we have $\varrho_p(f(w), f(z_n)) < \eta$ by Schwarz's lemma. Therefore
there exists a~constant $\delta^\star=\delta^\star(\delta,\eta)>0$ such that
\begin{equation*}
  \abs{ f'(w)} (1-\abs{w}^2) 
   \geq \delta^\star  \big( 1-\abs{f(z_n)}^2 \big)
   \geq  \delta^\star  \abs{ f'(z_n)} (1-\abs{z_n}^2)
  \geq \delta^\star \delta,
   \quad w\in\Delta(z_n,\eta),
\end{equation*}
by the Schwarz-Pick lemma. The claim follows for $\nu=\delta^\star\delta$.
\end{proof}


\begin{proof}[Proof of Theorem~\ref{thm:intp}]
The proof is divided into two steps. The first step takes advantage of two results concerning 
interpolation in $H^\infty$.


\medskip
\noindent\emph{Construction of auxiliary functions.}
Let $B=B_\Lambda$ be the Blaschke product corresponding to the uniformly separated sequence $\Lambda=\{z_n\}$,
and let $\{w_n\}$ be the bounded target sequence for the desired interpolation.
Consequently, 
\begin{equation} \label{eq:sep}
  \inf_{z_n\in\Lambda} \, |B'(z_n)| (1-|z_n|^2) 
  = \inf_{z_n\in\Lambda} \, \prod_{z_k\in\Lambda\setminus \{z_n\}} \left| \frac{z_k-z_n}{1-\overline{z}_kz_n}\right| =\delta >0.
\end{equation}

Let $0<\eta<1$ satisfy $12\eta/(1-\eta)^2< \delta/2$. Then, in particular, $\eta<\delta/3$.
Earl's interpolation theorem \cite[Theorem~2]{E:1970}, applied with $\eta$ instead of $\delta$, shows that 
\begin{equation} \label{eq:intprop}
\big\{ h \in H^\infty : \text{$h(z_n)=w_n$ for all $z_n\in\Lambda$} \big\}
\end{equation}
can be solved by a~constant multiple of a~Blaschke product. More precisely,
there exist $C=C(\Lambda, \{w_n\}, \eta) \in \C$ and 
a~Blaschke product $I=I(\Lambda, \{w_n\}, \eta)$ such that
\begin{enumerate}
\item[\rm (i)]
$h=C I$ solves the interpolation problem \eqref{eq:intprop};

\item[\rm (ii)]
the zeros $\Lambda^\star= \{\zeta_n\}$ of $I=I_{\{\zeta_n\}}$ satisfy $\zeta_n\in \Delta_p(z_n,\eta)$ for all $n$.
\end{enumerate}
The standard estimates show that 
\begin{equation*} 
 \inf_{\zeta_n\in\Lambda^\star} \, |I'(\zeta_n)| (1-|\zeta_n|^2) =
 \inf_{\zeta_n\in\Lambda^\star} \, \prod_{\zeta_k\in\Lambda^\star\setminus \{\zeta_n\}} 
 \left| \frac{\zeta_k-\zeta_n}{1-\overline{\zeta}_k \zeta_n}\right|  \geq \frac{\delta}{3} >0,
\end{equation*}
and therefore $\{\zeta_n\}$ is also uniformly separated. 

By applying Lemma~\ref{lemma:a2} to the Blaschke product $B$, there exists
another constant~$\nu$ such that $| B'(\zeta_n) | (1-|\zeta_n|^2) \geq  \nu >0$ for all $\zeta_n\in\Lambda^\star$.
According to \O yma's interpolation theorem \cite[Theorem~1]{O:1979}, there exists $g\in H^\infty$ such that
\begin{equation} \label{eq:int_g}
g(\zeta_n) = - \frac{I''(\zeta_n)}{2 \, I'(\zeta_n) \, B'(\zeta_n)}, \quad g'(\zeta_n)=0, \quad \zeta_n\in\Lambda^\star.
\end{equation}
Note that the target sequence for $g$ is bounded by the obtained estimates.


\medskip
\noindent\emph{Construction of the differential equation.}
Let $f=C I \, e^{Bg} \in\H(\D)$, where $C\in\C$ and $I,B,g$ are functions as in the above construction.
Clearly, $f(z_n)=w_n$ for all $z_n\in\Lambda$. The zeros of $f$ are precisely the points in $\Lambda^\star$, and they
are pairwise pseudo-hyperbolically close to the corresponding points in $\Lambda$ by (ii). Since
\begin{equation*}
f''(\zeta_n) = C e^{B(\zeta_n) g(\zeta_n)} \Big( I''(\zeta_n) + 2 \, I'(\zeta_n) B'(\zeta_n) g(\zeta_n) \Big) =0, \quad \zeta_n\in\Lambda^\star,
\end{equation*}
the function
\begin{equation*}
A = -\frac{f''}{f} = - \frac{I''+2\,  I' (Bg)'}{I} - \big( (Bg)' \big)^2 - (Bg)'',
\end{equation*}
is analytic in $\D$. More precisely, the points in $\Lambda^\star$ are removable singularities for 
the coefficient $A$ as
$g$ solves the interpolation problem \eqref{eq:int_g}.
As in the proof of Theorem~\ref{thm:converse_final}, 
we conclude that $|A(z)|^2(1-|z|^2)^3\, dm(z)$ is a~Carleson measure. The fact that
all solutions of \eqref{eq:de2} are bounded follows as
in the proof of Theorem~\ref{thm:steinmetz}.
This completes the proof of Theorem~\ref{thm:intp}.
\end{proof}


\subsection*{Separation of zeros and critical points}

Let $A\in H^\infty_2$, and let $f$ be a~non-trivial solution of \eqref{eq:de2}.
By \cite[Theorem~3]{S:1955}, the zeros of $f$
are separated in the hyperbolic metric by a~constant depending only on $\nm{A}_{H^\infty_2}$,
and by \cite[Corollary~2]{G:2017_1}, the hyperbolic distance between 
any zero and any critical point of $f$ is uniformly bounded away from zero 
in a~similar fashion. Moreover,
\cite[Example~1]{G:2017_1} shows that critical points of $f$ need not to obey any kind of
separation. The situation becomes more difficult if we consider similar questions between
zeros and critical points of linearly independent solutions.
See \cite[Section~4]{G:2017_1} for related discussion. 

The following result concerns differential equations with bounded
solutions. The proof is based on an auxiliary estimate \cite[Lemma~7, p.~209]{DS:2004}:
If $f\in H^\infty_\alpha$ for $0\leq \alpha< \infty$, then there exists a~constant $C=C(\alpha)$ 
with $0<C<\infty$ such that
\begin{equation} \label{eq:ds}
  \big| |f(z_1)|(1-|z_1|^2)^\alpha - |f(z_2)|(1-|z_2|^2)^\alpha \big|
  \leq C \, \varrho_p(z_1,z_2) \nm{f}_{H^\infty_\alpha}
\end{equation}
for all points $z_1,z_2\in\D$ with $\varrho_p(z_1,z_2)\leq 1/2$.
The sharpness discussion of Proposition~\ref{prop:sepa} below is omitted.


\begin{proposition} \label{prop:sepa}
Suppose that $A\in\H(\D)$ and all solutions of \eqref{eq:de2} are bounded.
\begin{itemize}
\item[\rm (i)]
It is possible that for each $0<\delta<1$ there exists a~solution of \eqref{eq:de2}, depending on $\delta$, 
which has two distinct zeros $z_1,z_2\in\D$ such that $\varrho_p(z_1,z_2) <\delta$.

\item[\rm (ii)]
Critical points of non-trivial solutions are not separated in any way.
\end{itemize}
Let $f_1,f_2\in H^\infty$ be linearly independent solutions of \eqref{eq:de2}.
\begin{itemize}
\item[\rm (iii)]
If $z_1\in\D$ is a~zero and $z_2\in\D$ is a~critical point of $f_1$,
then there exists a~constant $0<C<\infty$ such that
\begin{equation} \label{eq:sepa_old}
\varrho_p(z_1,z_2) \geq C \, \frac{|W(f_1,f_2)|}{\nm{f_1}_{H^\infty}  \nm{f_2}_{H^\infty}} \, \max\!\big\{1-|z_1|, 1-|z_2|\big\}.
\end{equation}

\item[\rm (iv)]
If $z_1\in\D$ is a~zero of $f_1$, and $z_2\in\D$ is a~zero of $f_2$, then \eqref{eq:sepa_old} holds.

\item[\rm (v)]
If $z_1\in\D$ is a~critical point of $f_1$, and $z_2\in\D$ is that of $f_2$, then \eqref{eq:sepa_old} holds.
\end{itemize}
\end{proposition}


\begin{proof}
(i) Let the coefficient $A\in\H(\D) \setminus H^\infty_2$ be as in Theorem~\ref{thm:converse}(i).
If the pseudo-hyperbolic distance between any distinct zeros of any non-trivial solution of \eqref{eq:de2}
is uniformly bounded away from zero, then $A\in H^\infty_2$ by \cite[Theorem~4]{S:1955}.
This is a~contradiction, and therefore (i) holds in this particular case.

(ii) The assertion follows from \cite[Example~1]{G:2017_1}, since 
in this example all solutions of \eqref{eq:de2} are bounded; use \eqref{eq:second} to obtain a~bounded linearly independent solution.

(iii) Let $f_1\in H^\infty$ be the non-trivial solution of \eqref{eq:de2} with $f_1(z_1)=0=f_1'(z_2)$.
If $\varrho_p(z_1,z_2) > 1/2$, then there is nothing to prove. Otherwise, let $f_2\in H^\infty$
be a~solution  of \eqref{eq:de2}, which is linearly independent to $f_1$.  
Since $f_1(z_2)f_2'(z_2) = W(f_1,f_2)$, there exists a~constant $0<C_1<\infty$ such that
\begin{equation*}
\varrho_p(z_1,z_2) \geq \frac{|f_1(z_2)|}{C_1 \, \nm{f_1}_{H^\infty}} = \frac{|W(f_1,f_2)|}{C_1 \, \nm{f_1}_{H^\infty} |f_2'(z_2)|}
\geq \frac{|W(f_1,f_2)|(1-|z_2|^2)}{C_1 \, \nm{f_1}_{H^\infty} \nm{f_2}_{H^\infty}}
\end{equation*}
by \eqref{eq:ds}; note that $\nm{f_2'}_{H^\infty_1}\leq \nm{f_2}_{H^\infty}$ by standard estimates.
Analogously, since $- f_1'(z_1)f_2(z_1) = W(f_1,f_2)$, there exists another constant $0<C_2<\infty$ such that
\begin{equation*}
  \varrho_p(z_1,z_2) 
  \geq \frac{ |f_1'(z_1)|(1-|z_1|^2)}{C_2 \, \nm{f_1'}_{H^\infty_1}} 
  = \frac{|W(f_1,f_2)|(1-|z_1|^2)}{C_2 \, \nm{f_1'}_{H^\infty_1} |f_2(z_1)|}
\geq \frac{|W(f_1,f_2)|(1-|z_1|^2)}{C_2 \, \nm{f_1}_{H^\infty} \nm{f_2}_{H^\infty}}.
\end{equation*}

Statements (iv) and (v) are proved similarly. In the case of (iv) apply \eqref{eq:ds}
to $f_1,f_2\in H^\infty$, and in the case of (v) apply \eqref{eq:ds} to $f_1',f_2'\in H^\infty_1$.
\end{proof}


\section{Proof of Theorem~\ref{thm:repres}} \label{sec:6}

After the proof of Theorem~\ref{thm:repres}, we consider its relation to conformal
metrics of constant curvature. We also discuss an~application concerning Carleson
measures induced by bounded solutions of \eqref{eq:de2} for $A\in\H(\D)$.


\begin{proof}[Proof of Theorem~\ref{thm:repres}]
It is clear that $u$ is sufficiently smooth to be in the class $C^2$.

(i) Since $(f_1/f_2)'=-W(f_1,f_2)/f_2^2$, we deduce
\begin{equation*} 
\left( \frac{f_1}{f_2} \right)^\# = \frac{|W(f_1,f_2)|}{|f_1|^2+|f_2|^2},
\qquad
\partial u =  \frac{f_1' \overline{f_1} + f_2' \overline{f_2}}{|f_1|^2+|f_2|^2}.
\end{equation*}
We compute
\begin{equation*}
  \Delta u 
   = 4 \, \overline{\partial}(\partial u) 
   = 4 \, \frac{|f_1f_2'-f_1'f_2|^2}{(|f_1|^2+|f_2|^2)^2} 
   = 4 \, e^{-2u}.
\end{equation*}

(ii) As above, we obtain
\begin{align*}
  \frac{1}{4} \, \Delta u & = \partial \big( \overline{\partial} u \big)
    = \frac{|f_1'|^2+|f_2'|^2}{|f_1|^2+|f_2|^2} - 
    \frac{f_1\overline{f_1'} + f_2 \overline{f_2'}}{|f_1|^2+|f_2|^2} \cdot \frac{f_1'\overline{f_1} + f_2' \overline{f_2}}{|f_1|^2+|f_2|^2}\\
  & = \frac{|f_1'|^2+|f_2'|^2}{|f_1|^2+|f_2|^2} - \big( \overline{\partial}u \big) \cdot \big( \partial u \big).
\end{align*}
Since $u$ is real-valued, $\Delta u  =  (\Delta e^u)/(e^u) - 4 \, |\partial u|^2 = (\Delta e^u)/(e^u) - |\nabla u|^2$.

(iii) We deduce
\begin{equation*}
\partial^2 u =  \, \frac{f_1'' \overline{f}_1 + f_2'' \overline{f}_2}{|f_1|^2+|f_2|^2} - (\partial u)^2
    = - \, \frac{A |f_1|^2 + A |f_2|^2}{|f_1|^2+|f_2|^2} - (\partial u)^2 = - A - (\partial u)^2,
\end{equation*}
which completes the proof.
\end{proof}


\begin{remark} \label{remark:conformalmetric}
Let $f_1$ and $f_2$ be linearly independent solutions of \eqref{eq:de2} for $A\in\H(\D)$.
As in the proof of Theorem~\ref{thm:repres}(i), we deduce that $v= -u =  \log\, (f_1/f_2)^{\#}$ is a~solution
of the Liouville equation $\Delta v = -4 \, e^{2 v}$. 
Recall that $\lambda(z)|dz|$ is said to be a~conformal metric on $\D$ if
the conformal density $\lambda: \D \to \R$ is strictly positive and continuous.
If $\lambda\in C^2$, then $\lambda(z)|dz|$ is called a~regular conformal metric on~$\D$. 
The (Gauss) curvature $\kappa: \D \to \R$ 
of the regular conformal metric $\lambda(z)|dz|$ is given
by $\kappa = - \Delta(\log \lambda)/\lambda^2$.
In conclusion, $(f_1/f_2)^{\#}(z)|dz|$ defines a~regular conformal metric 
of constant curvature $4$ on $\D$.
\end{remark}

As an~application of Theorem~\ref{thm:repres}, we return to consider differential 
equations with bounded solutions.
Theorem~\ref{thm:steinmetz} shows that, even if $f_1,f_2\in H^\infty$ are linearly
independent solutions of \eqref{eq:de2} for $A\in\H(\D)$, it may happen that $f_1/f_2$ is non-normal
and $((f_1/f_2)^\#)^2 \log(1/|z|) \, dm(z)$ is not a~Carleson measure.
The following result and Theorem~\ref{thm:repres}(ii) imply that
this Carleson measure condition becomes true if the exponent $2$ is replaced by any
smaller value.


\begin{theorem} \label{thm:uchiyama}
Let $f_1,f_2\in H^\infty$ be linearly independent solutions of \eqref{eq:de2} for $A\in\H(\D)$.
Then, $(|f_1'|^2+|f_2'|^2)(|f_1|^2 + |f_2|^2)^{\varepsilon-1} \log(1/|z|) \, dm(z)$
is a~Carleson measure for any $0<\varepsilon<\infty$.
\end{theorem}

Let $\Omega\subset \R^2$ be a~domain with smooth boundary, and let
$u_1,u_2$ be $C^2$-functions on $\overline{\Omega}$. The classical Green theorem
asserts
\begin{equation} \label{eq:green_origo}
  \int_{\partial \Omega} \left( u_1 \, \frac{\partial u_2}{\partial n} - u_2 \, \frac{\partial u_1}{\partial n} \right)  ds
   = \int_{\Omega} \big( u_1 \, \Delta u_2 - u_2 \, \Delta u_1 \big) \, dxdy,
\end{equation}
where $\partial/\partial n$ denotes differentiation in the direction of outward pointing normal
and $ds$ is the arc length on $\partial \Omega$.
The following argument is based on a~modification of Uchiyama's lemma.
We refer to \cite[p.~290]{N:1985} and \cite[Lemma~2.1]{TW:2005} for the original statement.
Suppose that $f\in\H(\D)$ and $\varphi\in C^2$ is a~subharmonic function in $\D$.
By the theorems of Green and Fubini, we deduce
\begin{align*}
&\frac{1}{2\pi} \int_0^{2\pi}  e^{\varphi(re^{i\theta})} |f(re^{i\theta})|^2 \, d\theta - e^{\varphi(0)} |f(0)|^2
      = \frac{1}{2\pi} \int_{D(0,r)} \Delta( e^{\varphi} |f|^2)(z)  \log\frac{r}{|z|} \, dm(z)
\end{align*} 
for any $0<r<1$. Since $(x-y)^2 \geq x^2/2 - y^2$ for $x,y\in\R$, we obtain
\begin{equation*} 
  \Delta ( e^\varphi |f|^2 ) 
  = e^\varphi (\Delta \varphi) |f|^2 + 4 \, e^\varphi | (\partial\varphi) f + f' |^2
  \geq  e^\varphi \Big( \Delta \varphi + 2 \, |\partial\varphi|^2 \Big) |f|^2 -  4 \, e^\varphi |f'|^2,
\end{equation*}
and further,
\begin{equation} \label{eq:uchi}
\begin{split}
& \frac{1}{2\pi} \int_{D(0,r)} |f(z)|^2 e^{\varphi(z)} \Big( \Delta \varphi(z) + 2 \, |\partial\varphi(z)|^2 \Big) \log\frac{r}{|z|} \, dm(z)\\
& \qquad \leq \frac{1}{2\pi} \int_0^{2\pi}  e^{\varphi(re^{i\theta})} |f(re^{i\theta})|^2 \, d\theta 
+ \frac{2}{\pi} \int_{D(0,r)} e^{\varphi(z)} |f'(z)|^2  \log\frac{r}{|z|} \, dm(z)
\end{split}
\end{equation}
for any $0<r<1$.


\begin{proof}[Proof of Theorem~\ref{thm:uchiyama}]
Let $f_1,f_2\in H^\infty$ be linearly independent solutions of \eqref{eq:de2} for $A\in\H(\D)$.
Without loss of generality, we may assume that $W(f_1,f_2)=1$.
We conclude that $\varphi=  \varepsilon \, u = \varepsilon \, \log ( |f_1|^2+|f_2|^2 )$
is bounded from above and subharmonic in $\D$, as $\Delta \varphi = 4 \, \varepsilon \, ( (f_1/f_2)^\# )^2\geq 0$ 
by Theorem~\ref{thm:repres}(i). 
By the Littlewood-Paley formula \cite[Lemma~3.1]{G:2007}, we obtain
\begin{equation*}
\nm{f}_{H^2}^2 = |f(0)|^2 + \frac{2}{\pi} \int_{\D} |f'(z)|^2 \log\frac{1}{|z|} \, dm(z), \quad f\in\H(\D),
\end{equation*}
and therefore a~standard convergence argument applied to \eqref{eq:uchi} reveals that
\begin{equation*}
\frac{1}{2\pi} \!\int_{\D} |f(z)|^2 e^{\varphi(z)} \Big( \Delta \varphi(z) + 2 \, |\partial\varphi(z)|^2 \Big)\log\frac{1}{|z|}  dm(z)
\leq 2 \big( \nm{f_1}_{H^\infty}^2 + \nm{f_2}_{H^\infty}^2 \big)^\varepsilon \nm{f}_{H^2}^2
\end{equation*}
for any $f\in H^2$. This proves that $e^{\varphi(z)} ( \Delta \varphi(z) + 2 \, |\partial\varphi(z)|^2 ) \log(1/|z|) \, dm(z)$
is a~Carleson measure, and therefore by Theorem~\ref{thm:repres}(ii), we deduce
\begin{equation*}
\begin{split}
  e^\varphi \big( \Delta \varphi + 2 \, |\partial\varphi|^2 \big)
 & = \big( |f_1|^2+|f_2|^2 \big)^\varepsilon \, \left( \varepsilon \, \Delta u + \frac{\varepsilon^2}{2} \, |\nabla u|^2 \right)\\
& \geq  \min\!\big\{ \varepsilon, \varepsilon^2/2 \big\} \big( |f_1|^2+|f_2|^2 \big)^\varepsilon \, \frac{\Delta e^u}{e^u}.
\end{split}
\end{equation*}
This completes the proof of Theorem~\ref{thm:uchiyama}.
\end{proof}


\section{Proofs of Theorem~\ref{theorem:new2} and Proposition~\ref{prop:eset}} \label{sec:nevrem}

Recall that the meromorphic
function $g$ in the unit disc belongs to the Nevanlinna class $\mathcal{N}$ if and only if the Ahlfors-Shimizu characteristic
\begin{equation*}
T_0(r,g) = \frac{1}{\pi} \, \int_0^r \bigg( \int_{D(0,t)} g^{\#}(z)^2\, dm(z) \bigg) \frac{dt}{t}
= \frac{1}{\pi} \, \int_{D(0,r)} g^{\#}(z)^2\, \log\frac{r}{|z|} \, dm(z) 
\end{equation*}
is uniformly bounded for $0<r<1$. The equivalence of the representations above follows from
Fubini's theorem.

Let $u\not\equiv -\infty$ be subharmonic in $\D$.
Function $u$ admits a~harmonic majorant in~$\D$ if and only if
$\lim_{r\to 1^-}  \int_0^{2\pi} u(re^{i\theta}) \, d\theta < \infty$,
and in this case, the least harmonic majorant for $u$ is
\begin{equation*}
  \hat{u}(z) 
  = \lim_{r\to 1^-} \, \frac{1}{2\pi}  \int_0^{2\pi} u(re^{i\theta}) \, 
  \frac{r^2-|z|^2}{|re^{i\theta}-z|^2}\, d\theta < \infty,
  \quad z\in\D.
\end{equation*}
See \cite[Theorem~3.3]{RR:1994} for more details. 
In the proof of Theorem~\ref{theorem:new2} we take advantage of the following
well-known fact: If $u\in C^2$ is subharmonic and $\phi$ is analytic, then $u\circ \phi$ is
subharmonic with $\Delta(u\circ \phi) = ((\Delta u) \circ \phi) \, |\phi'|^2$.


\begin{proof}[Proof of Theorem~\ref{theorem:new2}]
(i) By Green's theorem \eqref{eq:green_origo} with $u_1 = 1$, $u_2 = u$, we obtain
\begin{equation*}
  \frac{d}{dt} \int_0^{2\pi} u(te^{i\theta}) \, d\theta = \frac{4}{t} \, \int_{D(0,t)} \big( (f_1/f_2)^{\#}(z) \big)^2 \, dm(z),
\quad 0<t<1,
\end{equation*}
as $\Delta u = 4 \, ( (f_1/f_2)^{\#})^2$ by Theorem~\ref{thm:repres}(i). By integrating from 0 to $r$, we conclude
$1/(2\pi) \int_0^{2\pi} u(re^{i\theta}) \, d\theta
 = u(0) + 2 \, T_0\big( r,f_1/f_2 \big)$
for any $0<r<1$. Consequently, $u$ admits a~harmonic majorant if and only if $f_1/f_2\in\NC$.

(ii) Let $a\in\D$. By Green's theorem and Theorem~\ref{thm:repres}(i),
\begin{equation*} 
\begin{split}
&  \frac{1}{2\pi} \int_0^{2\pi} u\big(a+(1-|a|) \,re^{i\theta}\big) \, d\theta - u(a)\\
& \qquad   =  \frac{2}{\pi} \, \int_0^r \bigg( \int_{D(a,t(1-|a|))} (f_1/f_2)^{\#}(z)^2\, dm(z) \bigg) \frac{dt}{t},
  \quad 0<r<1.
\end{split}
\end{equation*}
By letting $r\to 1^-$, we deduce
\begin{equation} \label{eq:lp}
\sup_{a\in\D} \, \widehat{u_a}(0) = \sup_{a\in\D}\, \frac{2}{\pi} \, \int_0^1 \bigg( \int_{D(a,t(1-|a|))} (f_1/f_2)^{\#}(z)^2\, dm(z) \bigg) \frac{dt}{t}.
\end{equation}
This completes the proof of (ii), as $f_1/f_2$ is a~normal function in the Nevanlinna class if and only if
the right-hand side of \eqref{eq:lp} is finite \cite[Theorem~1]{P:1991}.

(iii) The assertion is in some sense a~meromorphic counterpart of \cite[Theorem~5.1]{Y:1982}.
Fix $a\in\D$, and take $r$ to be sufficiently large to satisfy $|a|<r<1$.
Define $\psi(z) = r \, \varphi_{a/r}(z/r)$, $z\in\D$. By Green's theorem,
\begin{align*}
\frac{1}{2\pi} \int_0^{2\pi} u(re^{i\theta}) \, \frac{r^2 - |a|^2}{|re^{i\theta}-a|^2}\, d\theta - u(a)
& = \frac{1}{2\pi} \int_0^{2\pi} (u \circ \psi)(re^{it}) \, dt - (u\circ\psi)(0)\\
& =\frac{1}{2\pi} \int_{D(0,r)} \Delta u(z)  \log\frac{1}{| \varphi_{a/r}(z/r) |} \, dm(z).
\end{align*}
By using standard estimates and letting $r\to 1^-$, we conclude that
$\hat{u}(a) - u(a) \asymp \int_{\D} \Delta u(z) ( 1 - |\varphi_{a}(z)|^2 ) \, dm(z)$,
where the comparison constants are independent of $a\in\D$. Theorem~\ref{thm:repres}(i) implies
\begin{equation*}
  \sup_{a\in\D} \big( \hat{u}(a) - u(a) \big) 
    \asymp \, \sup_{a\in\D} \,\int_{\D} \big( (f_1/f_2)^\#(z) \big)^2 (1-|z|^2) \, \frac{1-|a|^2}{|1-\overline{a}z|^2} \, dm(z).
\end{equation*}
The part (iii) follows as $f_1/f_2\in \rm{UBC}$ if and only if $( (f_1/f_2)^\#(z))^2 (1-|z|^2)\, dm(z)$ is a~Carleson measure
\cite[Theorem~3]{P:1991}.

The proofs of (iv)-(vi) are straight-forward and hence omitted.
Note that the function $e^u = (|f_1|^2+|f_2|^2)/|W(f_1,f_2)|$ is subharmonic in $\D$.
\end{proof}

It is well-known that non-trivial solutions of a~Blaschke-oscillatory equation~\eqref{eq:de2}, $A\in\H(\D)$,
may lie outside the Nevanlinna class $\NC$ \cite[Section~4.3]{H:2013}. In the following
remark, we deduce an~estimate according to which the Nevanlinna characteristic of solutions
of Blaschke-oscillatory equations cannot grow arbitrarily~fast.


\begin{remark} \label{remark:boder}
Let $f_1$ be a~non-trivial solution of a~Blaschke-oscillatory equation \eqref{eq:de2} for $A\in\H(\D)$.
Let $f_2$ be another solution of \eqref{eq:de2}, which is
linearly independent to $f_1$. 
Note that $f_2/f_1 \in \NC$ by \cite[Lemma~3]{H:2013}, and $(f_2/f_1)'=W(f_1,f_2)/f_1^2$
by straight-forward computation.
Kennedy's estimate \cite[Theorem~1]{K:1964} implies
\begin{equation} \label{eq:kennedy}
S = \int_0^1 (1-r) e^{2 \, T(r,(f_2/f_1)')} \, dr < \infty.
\end{equation}
Nevanlinna's first theorem shows that \eqref{eq:kennedy} remains to be true,
if $T(r,(f_2/f_1)')$ is replaced by $2 \, T(r,f_1)$. This places a~severe restriction for the
growth of $T(r,f_1)$ as $r\to 1^-$. Among other things, it implies that
$T(r,f_1) \leq (1/2) \log(\sqrt{2S}/(1-r))$ for all $0<r<1$. Therefore all solutions of \eqref{eq:de2} are
non-admissible \cite[p.~53]{H:2013}.
\end{remark}


\begin{proof}[Proof of Proposition~\ref{prop:eset}]
Recall that $(f_1/f_2)^\#= |W(f_1,f_2)|/(|f_1|^2+|f_2|^2)$. Now
\begin{align*}
& \int_{\D} \big( ( f_1/f_2 )^\# \big)^2 (1-|z|^2)\, dm(z)\\
& \qquad \leq \frac{|W(f_1,f_2)|^2}{\delta^2} \int_{\{z\in\D \; \! : \; \! |f_1(z)|^2+|f_2(z)|^2 \geq \delta\}} (1-|z|^2)\, dm(z)\\
& \qquad \qquad + \left( \, \sup_{z\in\D} \, ( f_1/f_2 )^\#(z)^2 (1-|z|^2)^2 \right)  \int_{\{z\in\D \; \! : \; \! |f_1(z)|^2+|f_2(z)|^2 < \delta\}} \frac{dm(z)}{1-|z|^2}.
\end{align*}
Therefore $f_1/f_2$ belongs to the Nevanlinna class by the assumption.
\end{proof}

We briefly consider two applications of Proposition~\ref{prop:eset}.
Suppose that $f_1,f_2$ are linearly independent solutions of \eqref{eq:de2} for $A\in\H(\D)$ 
and assume that \eqref{eq:exc_ind} holds for some Blaschke sequence $\{z_n\}\subset \D$ and
$0<\delta<1$. Denote this infimum by $0<s<\infty$. We deduce
\begin{align*}
 \int_{\{z\in\D \; \! : \; \! |f_1(z)|^2+|f_2(z)|^2 < s^2/2\}} \frac{dm(z)}{1-|z|^2}
   \leq \sum_n \int_{\Delta_p(z_n,\delta)} \frac{dm(z)}{1-|z|^2} \asymp \sum_n (1-|z_n|) < \infty,
\end{align*}
where the pseudo-hyperbolic discs $\Delta_p(z_n,\delta)$ are not necessarily
pairwise disjoint. In such a case the normality of $f_1/f_2$ implies that $f_1/f_2\in\NC$ by Proposition~\ref{prop:eset}.

The same conclusion is obtained if $f_1/f_2$ is normal and $|f_1|+|f_2|$ is uniformly bounded from below for
all points in $\D$ which lie outside a~horodisc (that is, a~disc internally tangent to~$\partial\D$).
The details are omitted.


\section{Proof of Theorem~\ref{thm:n3}} \label{sec:n3}

We begin with a~lemma, which is needed in the proof of Theorem~\ref{thm:n3}. This
auxiliary result is based on the well-known Harnack inequalities: if $h\in \hp(\D)$,~then
\begin{equation*}
\frac{1-\varrho_p(z,w)}{1+\varrho_p(z,w)} \leq \frac{h(z)}{h(w)} \leq \frac{1+\varrho_p(z,w)}{1-\varrho_p(z,w)},
\quad z,w\in\D.
\end{equation*}

Let $f\in\H(\D)$ and recall that $f\in\NC$ if and only if there exists $h \in \hp(\D)$ such that
$\log^+ |f| \leq h$, which is equivalent to the fact $|f| \leq e^h$. There is no reason to expect that any order derivative of $f$
would belong to $\NC$. However, for every $k\in\N$, there exists a~constant $C=C(k)$ with $0<C<\infty$ such that
\begin{equation} \label{eq:nevder}
  |f^{(k)}(z)| (1-|z|^2)^k \leq e^{C\,  h(z)}, \quad z\in\D,
\end{equation}
by Cauchy's integral formula and Harnack's inequality.
See \cite[Lemma~2.1]{HMN:preprint}.


\begin{lemma} \label{lemma:betterest}
Suppose that $f\in\H(\D)$ and it satisfies 
$|f(z)| (1-|z|^2)^k \leq e^{h(z)}$, $z\in\D$,
for $k\in\N \cup \{0\}$ and $h\in\hp(\D)$. If $f$ vanishes on a~sequence $\Lambda\in\Int\NC$, 
then there exists $H\in\hp(\D)$ such that
$|f(z)| (1-|z|^2)^k \leq \varrho_p(\Lambda, z) \, e^{H(z)}$, $z\in\D$.
\end{lemma}


\begin{proof}
Consider a~dyadic partition of $\D$ into Whitney squares of the type
\begin{equation*} 
Q  =Q_I = \big\{ z\in\D : 1-\abs{I}/(2\pi) \leq |z| < 1, \, \arg{z}\in I \big\}
\end{equation*}
where $\ell(Q)=\abs{I}$ is the arc-length of the interval $I\subset \partial\D$.
The top part of $Q$ is 
$T(Q) = \{ z\in Q : 1-\ell(Q)/(2\pi) \leq |z| \leq 1-\ell(Q)/(4\pi) \}$.

Let $Q$ be any Whitney square in the dyadic partition. Let $\Omega_1\subset \D$ such that
\begin{equation*}
T(Q) \subset \Omega_1, \quad \varrho_p\big( \partial\Omega_1, \partial T(Q)\big) = \diam_p\!\big( T(Q) \big),
\end{equation*}
and let $\Omega_2$ be another set such that
$\Omega_1 \subset \Omega_2 \subset\D$ and $\varrho_p( \partial\Omega_2, \partial \Omega_1 ) = 4 \diam_p\Omega_1$.
Here $\diam_p$ denotes the pseudo-hyperbolic diameter.
Define $g\in\mathcal{H}(\D)$ by
\begin{equation*}
g(z) = f(z) \Bigg( \, \prod_{z_k \in  \Lambda \, \cap \, \Omega_1} \frac{z_k - z}{1-\overline{z}_k z} \Bigg)^{-1}, \quad z\in\D.
\end{equation*}
We may assume that $\Lambda\cap \Omega_1$ is not empty, for
otherwise the assertion follows for all $z\in T(Q)$ by trivial reasons.
Fix any $z_n \in  \Lambda \, \cap \, \Omega_1$. We deduce
\begin{equation*}
|g(\zeta)| \leq \frac{(1-|\zeta|^2)^{-k} e^{h(\zeta)}}{\varrho_p(z_n,\zeta)} \Bigg( \, 
\prod_{z_k \in  \Lambda \, \cap \, \Omega_1 \, : \, z_k\neq z_n} 
\left| \frac{z_k - \zeta}{1-\overline{z}_k \zeta} \right| \Bigg)^{-1}, \quad \zeta\in\partial\Omega_2.
\end{equation*}
Since $\Lambda\in\Int\NC$, \cite[Theorem~1.2]{HMNT:2004} implies that
there exists $h_1\in \hp(\D)$ with
\begin{equation*}
  |g(\zeta)| 
  \lesssim (1-|\zeta|^2)^{-k} e^{h(\zeta)+h_1(z_n)} 
  \lesssim (1-|z_n|^2)^{-k} e^{(Ch+h_1)(z_n)}, \quad \zeta\in\partial\Omega_2,
\end{equation*}
where $0<C<\infty$ is a~universal constant by Harnack's inequalities.
The maximum modulus principle extends this estimate
for all $z\in\Omega_2$, and therefore
\begin{equation*}
  |f(z)| 
  \leq |g(z)| \prod_{z_k \in  \Lambda \, \cap \, \Omega_1} \left| \frac{z_k - z}{1-\overline{z}_k z} \right|
  \lesssim (1-|z_n|^2)^{-k} e^{(Ch+h_1)(z_n)} \, \varrho_p(\Lambda,z), \quad z\in T(Q).
\end{equation*}
By Harnack's inequalities, there exists $H\in\hp(\D)$ such that
the assertion holds for all $z\in T(Q)$.
Since the argument is independent of the Whitney square $Q$, the proof is complete.
\end{proof}


\begin{proof}[Proof of Theorem~\ref{thm:n3}]
Let $B=B_\Lambda$ be the Blaschke product with zeros $\Lambda \in\Int\NC$ and let
$f=Be^{Bg}$, where $g\in\H(\D)$ is a~solution of the interpolation problem
\begin{equation} \label{eq:nevint}
g(z_n) = w_n,  \quad w_n = - \frac{B''(z_n)}{2 \big(B'(z_n)\big)^2}, \quad z_n\in\Lambda.
\end{equation}
As $\Lambda\in\Int\NC$, \cite[Theorem~1.2]{HMNT:2004} implies that 
there exists $h_1\in\hp(\D)$ with
\begin{equation} \label{eq:nevii}
  |B'(z_n)| (1-|z_n|^2)  
  = \prod_{z_k \in\Lambda \, : \, z_k\neq z_n} \left| \frac{z_k - z_n}{1-\overline{z}_k z_n} \right|
  \geq e^{-h_1(z_n)}, \quad z_n\in\Lambda.
\end{equation}
Since there exists a~constant $0<C<\infty$ such that 
\begin{equation*}
\log^+ |w_n| = \log^+ \left| \frac{B''(z_n)}{2 \big(B'(z_n)\big)^2} \right| \leq  C + 2\, h_1(z_n), \quad z_n\in\D,
\end{equation*}
\cite[Theorem~1.2]{HMNT:2004} ensures that $\{w_n\} \in \NC\mid \Lambda$.
Therefore we may assume~$g\in\NC$.

By straight-forward computation, $f$ is a~solution of \eqref{eq:de2} for $A\in\H(\D)$, where
\begin{equation} \label{eq:arepp}
A = - \frac{f''}{f} = - \frac{B'' + 2 B' (B'g+Bg')}{B} - {\big( (Bg)' \big)}^2-(Bg)''.
\end{equation}
The interpolation property \eqref{eq:nevint} guarantees that every point $z_n\in\Lambda$ is
a~removable singularity for~$A$. It remains to show that there exists $h\in\hp(\D)$ such that
$|A(z)|(1-|z|^2)^2 \leq e^{h(z)}$, $z\in\D$. Since $Bg \in\NC$, \eqref{eq:nevder} 
implies that  the two right-most terms in \eqref{eq:arepp} are of the desired type.
Since $B'' + 2 B' (B'g+Bg')$ vanishes on the sequence $\Lambda$, Lemma~\ref{lemma:betterest} shows
that there exists $h_2\in \hp(\D)$ such that 
\begin{equation*}
  \big|B''(z) + 2 B'(z) \big(B'(z)g(z)+B(z)g'(z)\big)\big| (1-|z|^2)^{2} \leq \varrho_p(\Lambda, z) \, e^{h_2(z)}, 
  \quad z\in\D.
\end{equation*}
And finally, by \cite[Theorem~1.2]{HMN:preprint}, 
there exists $h_3\in\hp(\D)$ such that $|B(z)| \geq \varrho_p(\Lambda, z) e^{-h_3(z)}$, $z\in \D$.
We deduce Theorem~\ref{thm:n3} by combining the estimates.
\end{proof}


\section{Proof of Theorem~\ref{thm:carleson}} \label{sec:carleson_a}

The following result is an analogue of Carleson's \cite[Theorem~2]{C:1962},
which characterizes those cases in which the classical $0,1$-interpolation is possible. The
proof of Proposition~\ref{prop:carleson} is based on the
\emph{Nevanlinna corona theorem} by Mortini \cite[Satz~4]{M:1989}: Given $f_1,f_2\in\NC$,
the B\'ezout equation $f_1g_2+f_2g_2 = 1$ can be solved with functions $g_1,g_2\in\NC$
if and only if there exists $h\in\hp(\D)$ such that $|f_1(z)|+|f_2(z)|\geq e^{-h(z)}$, $z\in\D$.


\begin{proposition} \label{prop:carleson}
Let $\{z_n\},\{\zeta_n\}$ be Blaschke sequences. Then, there exists $f\in\NC$ such that
$f(z_n)=0$ and $f(\zeta_n)=1$ for all $n$ if and only if there exists $h\in\hp(\D)$ such that
\eqref{eq:cca} holds.
\end{proposition}


\begin{proof}
Assume that there exists $f\in\NC$ such that $f(z_n)=0$ and $f(\zeta_n)=1$ for all~$n$.
By the classical factorization theorem, there exist functions $g_1,g_2\in\NC$ such that
$f=B_{\{z_n\}} g_1 = 1+B_{\{\zeta_n\}}g_2$. Here $B_{\{z_n\}}$ and $B_{\{\zeta_n\}}$ are Blaschke products
with zeros $\{z_n\}$ and $\{\zeta_n\}$, respectively. As $g_1,g_2\in\NC$, there exist $h_1,h_2\in\hp(\D)$ such that 
$|g_1| \leq e^{h_1}$ and $|g_2| \leq e^{h_2}$. We deduce
\begin{equation*}
1  = \big| B_{\{z_n\}} g_1 - B_{\{\zeta_n\}}g_2 \big| 
 \leq  e^{h_1+h_2} \big( | B_{\{z_n\}} | + | B_{\{\zeta_n\}} |\big),
\end{equation*}
which proves the first part of the assertion.

Assume that there exists $h\in\hp(\D)$ such that \eqref{eq:cca} holds.
By the Nevanlinna corona theorem, 
there exist $g_1,g_2\in\NC$ such that $B_{\{ z_n \}} g_1 + B_{\{\zeta_n\}} g_2=1$.
Then, the function $f=B_{\{ z_n \}} g_1 \in \NC$ satisfies the desired $0,1$-interpolation.
\end{proof}


\begin{proof}[Proof of Theorem~\ref{thm:carleson}]
By Proposition~\ref{prop:carleson}, there exists $g\in\NC$ such that
$g(z_n)=0$ and $g(\zeta_n)=1$ for all $n$. Now
$f(z)=\exp(\log \alpha + g(z) \log(\beta/\alpha))$, $z\in\D$,
satisfies the desired interpolation property, and is a~zero-free solution of \eqref{eq:de2} for $A\in\H(\D)$,
\begin{equation*}
A(z) = - \frac{f''(z)}{f(z)} = - \left( g'(z) \log\frac{\beta}{\alpha} \right)^2 - g''(z) \log\frac{\beta}{\alpha}, \quad z\in\D.
\end{equation*}
By \eqref{eq:nevder}, there exists $H\in\hp(\D)$ such that $|A(z)|(1-|z|^2)^2 \leq e^{H(z)}$,~$z\in\D$.
\end{proof}


\section{Proofs of Theorems~\ref{thm:n1} and \ref{thm:n2}, and Proposition~\ref{prop:n4}} 

The following proof proceeds along the same lines as that in \cite[p.~129]{S:2012}.


\begin{proof}[Proof of Theorem~\ref{thm:n1}]
If $f_1,f_2$ are linearly
independent solutions of \eqref{eq:de2} for $A\in\H(\D)$, then 
$W(f_1,f_2)\, A = f_1'f_2''-f_1''f_2'$ and
$(f_1/f_2)^\# = |W(f_1,f_2)|/(|f_1|^2+|f_2|^2)$. 
Since $W(f_1,f_2)$ is a~non-zero complex constant,
the estimate \eqref{eq:nevder} and the fact $f_1,f_2\in\NC$ imply that there exists $h_1\in \hp(\D)$ such that
$|A(z)|(1-|z|^2)^3 \leq e^{h_1(z)}$, $z\in\D$. Moreover, the Cauchy-Schwarz inequality \eqref{eq:csaa}
and the estimate \eqref{eq:nevder} show that there exists $h_2\in\hp(\D)$ such that 
$(f_1/f_2)^{\#}(z)(1-|z|^2)^2 \leq e^{h_2(z)}$, $z\in\D$. The claim follows by choosing $H=h_1+h_2 \in \hp(\D)$.
\end{proof}

The proof of Theorem~\ref{thm:n2} is analogous to that proof of Theorem~\ref{thm:converse_finalx},
which is presented in the end of Section~\ref{sec:ppp}.


\begin{proof}[Proof of Theorem~\ref{thm:n2}]
By \eqref{eq:toloN} and \cite[Theorem~1]{HMN:preprint}, the ideal $I_{\NC}(f_1,f_2)$ contains a~Blaschke product $B$
whose zero-sequence belongs to $\Int\NC$.
This is equivalent to the fact that there exist functions $g_1,g_2\in \NC$ such that $f_1 g_1 + f_2 g_2 = B$. 
Differentiate $f_1 g_1 + f_2 g_2 = B$ twice, and apply \eqref{eq:de2} to $f_1''$ and $f_2''$
to obtain \eqref{eq:diff_twice}. Note that $A\in\H(\D)$ by assumption.
As in the proof of Theorem~\ref{thm:n3}, 
we conclude that there exists $H\in\hp(\D)$ such that $\sup_{z\in\D} |A(z)|(1-|z|^2)^2 \leq e^{H(z)}$, $z\in\D$.
\end{proof}


\begin{proof}[Proof of Proposition~\ref{prop:n4}]
Proposition~\ref{prop:n4} follows directly from \cite[Theorem~15]{CGHR:2013} if 
$\psi : \D \to (0,1/2)$ given by $\psi(z) = e^{-H(z)/2}e^{-1}$, $z\in\D$ and $H\in\hp(\D)$, 
satisfies $\sup_{a,z\in\D} \, \psi(a)/\psi\big( \varphi_a(\psi(a)z) \big) < \infty$.
Now
\begin{align*}
&  \sup_{a,z\in\D} \,\exp\!\left( \frac{H(a)}{2} \left( \frac{H\big( \varphi_a(e^{-H(a)/2}e^{-1}z)\big)}{H(\varphi_a(0))} - 1 \right) \right)\\
& \qquad  \leq \sup_{a,z\in\D} \,\exp\!\left( \frac{H(a)}{2} 
    \left( \frac{1+\varrho_p\big(0, e^{-H(a)/2}e^{-1}z\big)}{1-\varrho_p\big(0,e^{-H(a)/2}e^{-1}z\big)} - 1 \right) \right)
\end{align*}
by Harnack's inequalities. This is bounded by
\begin{equation*}
\sup_{0\leq x < \infty} \exp\!\left( \frac{x}{2}
    \left( \frac{1+e^{-x/2}e^{-1}}{1-e^{-x/2}e^{-1}} - 1 \right) \right)< \frac{3}{2},
\end{equation*}
which implies the assertion.
\end{proof}


\subsection*{Separation of zeros and critical points}
We proceed to state an~analogue of Proposition~\ref{prop:sepa}.
If $f\in\H(\D)$ and
\begin{equation} \label{eq:nnm}
\nm{f} = \sup_{z\in\D} \, |f(z)| (1-|z|^2)^\alpha e^{-h(z)} < \infty
\end{equation}
for $0\leq \alpha< \infty$ and $h\in\hp(\D)$, then there exists $C=C(\alpha)>0$ such that
\begin{equation*} \label{eq:dsn}
\begin{split}
   \Big| |f(z_1)|(1-|z_1|^2)^\alpha e^{-h(z_1)} - |f(z_2)|(1-|z_2|^2)^\alpha e^{-h(z_2)} \Big|
   \leq C \, \varrho_p(z_1,z_2) \, \nm{f},
\end{split}
\end{equation*}
for all points $z_1,z_2\in\D$ with $\varrho_p(z_1,z_2)\leq 1/2$.
This estimate follows immediately from \eqref{eq:ds}: If $f\in\H(\D)$
satisfies \eqref{eq:nnm} for $0\leq \alpha< \infty$ and $h\in\hp(\D)$,
then \eqref{eq:ds} can be applied to $f e^{-h - i h^\star} \in H^\infty_\alpha$, where $h^\star$ is a~harmonic conjugate of $h$.


\begin{proposition} \label{prop:sepaN}
Let $f_1,f_2\in \NC$ be linearly independent solutions of \eqref{eq:de2} for $A\in\H(\D)$.
\begin{itemize}
\item[\rm (i)]
If $z_1\in\D$ is a~zero and $z_2\in\D$ is a~critical point of $f_1$, then there exists $h\in\hp(\D)$ such that 
\begin{equation} \label{eq:sepa}
\varrho_p(z_1,z_2) \gtrsim \max\Big\{ (1-|z_1|)e^{-h(z_1)}, (1-|z_2|)e^{-h(z_2)} \Big\}.
\end{equation}

\item[\rm (ii)]
If $z_1\in\D$ is a~zero of $f_1$, and $z_2\in\D$ is a~zero of $f_2$, then 
there exists $h\in\hp(\D)$ such that \eqref{eq:sepa} holds.

\item[\rm (iii)]
If $z_1\in\D$ is a~critical point of $f_1$, and $z_2\in\D$ is a~critical point of $f_2$, then 
there exists $h\in\hp(\D)$ such that \eqref{eq:sepa} holds.
\end{itemize}
\end{proposition}

The proof of Proposition~\ref{prop:sepaN} is omitted.


\section{Proofs of Theorem~\ref{theorem:c2} and Corollary~\ref{cor:c2}} \label{sec:c2}

The proof of Theorem~\ref{theorem:c2} is based on a~smoothness property,
which is considered first. Let $\omega$ be a radial weight on $\D$. Then,
\begin{equation*}
  \varrho_\omega(z_1,z_2) = \int_{\langle z_1, z_2 \rangle} \frac{|dz|}{\omega(z)},
  \quad z_1,z_2\in\D,
\end{equation*}
defines a~distance function. Here, we integrate along
the hyperbolic segment $\langle z_1,z_2 \rangle$ between the points $z_1,z_2\in\D$,
where the hyperbolic segment is a~closed subset of the corresponding hyperbolic geodesic. 
For $\omega(z)=1-|z|^2$, $z\in\D$,
the function~$\varrho_\omega$ reduces to the standard hyperbolic distance $\varrho_h$:
\begin{equation*}
  \varrho_h(z_1,z_2) = \frac{1}{2} \log\frac{1+\varrho_p(z_1,z_2)}{1-\varrho_p(z_1,z_2)},
  \quad \varrho_p(z_1,z_2)=\left| \frac{z_2-z_1}{1-\overline{z}_2z_1}\right|,
  \quad z_1,z_2\in\D.
\end{equation*}


\begin{lemma} \label{lemma:smoothness}
Let $f_1,f_2$ be linearly independent solutions of \eqref{eq:de2} for $A\in\H(\D)$,
and define $u = -\log \, (f_1/f_2)^{\#}$.
Let $\omega$ be a~radial weight. If 
\begin{equation} \label{eq:presmooth}
\sup_{z\in\D} \, |\nabla u(z)| \, \omega(z) \leq \Lambda < \infty,
\end{equation}
then
\begin{equation} \label{eq:smooth}
  e^{-\Lambda  \varrho_\omega(z_1,z_2)}
  \leq \frac{|f_1(z_1)|^2+|f_2(z_1)|^2}{|f_1(z_2)|^2+|f_2(z_2)|^2} 
  \leq e^{ \Lambda  \varrho_\omega(z_1,z_2)}, \quad z_1,z_2\in\D.
\end{equation}
Conversely, if \eqref{eq:smooth} holds for some~constant $0<\Lambda<\infty$,
then \eqref{eq:presmooth} is satisfied.
\end{lemma}


\begin{proof}
Assume that \eqref{eq:presmooth} holds.
Let $z_1,z_2\in\D$ be distinct points, and let $\gamma=\gamma(t)$, $0\leq t \leq 1$, be a~parametrization
of $\langle z_1,z_2 \rangle$. 
Schwarz's inequality and \eqref{eq:presmooth} imply
\begin{align*}
  \left| \log \frac{|f_1(z_1)|^2+|f_2(z_1)|^2}{|f_1(z_2)|^2+|f_2(z_2)|^2} \right|
  & = \big| u(z_1) - u(z_2) \big|
    \leq \left| \int_0^1 \nabla u (\gamma(t)) \cdot \gamma'(t)  \, dt \right| \\
& \leq \int_0^1 | \nabla u (\gamma(t)) | \, |\gamma'(t) | \, dt 
\leq \Lambda \, \varrho_\omega(z_1,z_2).
\end{align*}
From this estimate we deduce \eqref{eq:smooth}.

Assume that \eqref{eq:smooth} holds for some constant $0<\Lambda<\infty$. Fix $z_2\in\D$.
Since 
\begin{equation*}
\lim_{z_1\to z_2} \frac{|z_1-z_2|}{\varrho_h(z_1,z_2)} 
= \lim_{z_1\to z_2} \frac{\varrho_p(z_1,z_2)}{\frac{1}{2} \log\frac{1+\varrho_p(z_1,z_2)}{1-\varrho_p(z_1,z_2)}}
     \cdot |1-\overline{z}_1z_2| = 1-|z_2|^2, 
\end{equation*}
and
\begin{equation*}
\frac{|z_1-z_2|}{\varrho_h(z_1,z_2)}  \cdot \frac{1}{\max_{z\in \langle z_1,z_2 \rangle} \frac{1-|z|^2}{\omega(z)}} 
\leq\frac{|z_1-z_2|}{\varrho_\omega(z_1,z_2)} \leq  \frac{|z_1-z_2|}{\varrho_h(z_1,z_2)}  \cdot \frac{1}{\min_{z\in \langle z_1,z_2 \rangle} \frac{1-|z|^2}{\omega(z)}}
\end{equation*}
for any $z_1\in\D$, we conclude that $\lim_{z_1\to z_2} |z_1-z_2|/\varrho_\omega(z_1,z_2) = \omega(z_2)$
by the continuity of $\omega$. Therefore,
\begin{align*}
\left| \nabla u(z_2) \right|  \omega(z_2)
  =  \lim_{z_1\to z_2} 
  \left| \frac{u(z_1) - u(z_2)}{z_1-z_2} \right| \frac{|z_1-z_2|}{\varrho_\omega(z_1,z_2)}
 \leq \lim_{z_1\to z_2} \frac{\Lambda \, \varrho_\omega(z_1,z_2)}
  {\varrho_\omega(z_1,z_2)} = \Lambda.
\end{align*}
This completes the proof of Lemma~\ref{lemma:smoothness}.
\end{proof}

The following lemma is important for our cause due to the representation \eqref{eq:preli}.


\begin{lemma} \label{lemma:derprop}
Let $f_1,f_2$ be linearly independent solutions of \eqref{eq:de2} for $A\in\H(\D)$,
and define $u = -\log \, (f_1/f_2)^{\#}$.
Suppose that $\omega$ is a~regular weight which satisfies $\sup_{z\in\D}\, \omega(z)/(1-|z|) < \infty$.
If $|\nabla u| \in L^\infty_\omega$, then
\begin{equation*}
\frac{|f_1^{(j)}|+|f_2^{(j)}|}{|f_1| + |f_2|} \in L^\infty_{\omega^j}, \quad j\in\N.
\end{equation*}
\end{lemma}


\begin{proof}
By the assumption,
there exists a~positive constant $c$ such that 
the discs $\mathfrak{D}(z)= D(z,c \,\omega(z))$
satisfy $\mathfrak{D}(z) \subset D(z,(1-|z|)/2)$, $z\in\D$.
Let $\zeta\in \partial \mathfrak{D}(z)$. Since $\langle z, \zeta\rangle \subset D(z,(1-|z|)/2)$, 
a~straight-forward argument based on \eqref{eq:regular} reveals
\begin{equation*}
\varrho_\omega(z,\zeta) \lesssim \frac{1-|z|^2}{\omega(z)} \, \varrho_h(z,\zeta)
\lesssim \frac{1-|z|^2}{\omega(z)} \, \varrho_p(z,\zeta) \lesssim \frac{|z-\zeta|}{\omega(z)} = c.
\end{equation*}
Therefore $\sup_{z\in\D} \max_{\zeta\in \partial \mathfrak{D}(z)}  \varrho_\omega(z,\zeta) < \infty$.

By Cauchy's integral formula,
\begin{equation} \label{eq:cader}
\begin{split}
  |f_1^{(j)}(z)|+|f_2^{(j)}(z)| 
  & \leq 2 \, \max\big\{ |f_1^{(j)}(z)|, |f_2^{(j)}(z)| \big\} \\
   & \leq \left( \, \max_{\zeta\in \partial \mathfrak{D}(z)} 
        \big( |f_1(\zeta)| + |f_2(\zeta)| \big) \right) \, \frac{2 j!}{c^j \, \omega(z)^j}, \quad z\in\D.
\end{split}
\end{equation}
Now \eqref{eq:cader} and Lemma~\ref{lemma:smoothness} imply
\begin{align*}
\frac{|f_1^{(j)}(z)|+|f_2^{(j)}(z)|}{|f_1(z)| + |f_2(z)|} 
&  \leq \frac{2 j! \sqrt{2}}{c^j \, \omega(z)^j}  \!\left( \max_{\zeta\in \partial \mathfrak{D}(z)}  
  \frac{|f_1(\zeta)|^2 + |f_2(\zeta)|^2}{|f_1(z)|^2 + |f_2(z)|^2} \right)^{1/2}\\
&  \leq \frac{2 j! \sqrt{2}}{c^j \, \omega(z)^j} \, 
  \exp\left( \frac{\nm{|\nabla u| }_{L^\infty_\omega}}{2} \max_{\zeta\in \partial \mathfrak{D}(z)} \, \varrho_\omega(z,\zeta) \right)
  \lesssim \frac{1}{\omega(z)^j}
\end{align*}
for $z\in\D$. The assertion of Lemma~\ref{lemma:derprop} follows.
\end{proof}

Finally, proceed to prove Theorem~\ref{theorem:c2}.
We take advantage of Yamashita's \cite[Corollary to Theorem~2, p.~161]{Y:2002}, 
which uses the following notation.
For a~meromorphic function $f$ and $z\in\D$, let $\rho(z,f)$ be the maximum
of $0<r\leq  1$ such that $f$ is univalent in $\Delta_p(z,r)$, and let
$\rho_a(z,f)$ be the maximum of $0<r\leq 1$ such that $f(w)\neq -1/\overline{f(z)}$,
which is the antipodal point of $f(z)$ in the Riemann sphere.


\begin{proof}[Proof of Theorem~\ref{theorem:c2}]
First, assume that $|\nabla u|\in L^\infty_\omega$.
By the representation \eqref{eq:preli} and Lemma~\ref{lemma:derprop}, we conclude that $A\in H^\infty_{\omega^2}$.
By Theorem~\ref{thm:repres}(i) and (ii),
\begin{equation*}
  4 \, \Big( \big( f_1/f_2\big)^\# \Big)^2 
  = \Delta u 
  \leq \frac{\Delta e^u}{e^u} 
  = \frac{|f_1'|^2+|f_2'|^2}{|f_1|^2+|f_2|^2} 
  \leq 2 \left( \frac{|f_1'|+|f_2'|}{|f_1|+|f_2|} \right)^2,
\end{equation*}
and therefore $(f_1/f_2)^\# \in L^\infty_\omega$ by Lemma~\ref{lemma:derprop}.

Second, let $A \in H^\infty_{\omega^2}$ and $(f_1/f_2)^\# \in L^\infty_\omega$. 
Since $f_1/f_2$ is meromorphic in $\D$ and has zero-free spherical derivative,
Yamashita's \cite[Corollary to Theorem~2, p.~161]{Y:2002} implies
\begin{equation*}
(1-|z|^2) \left| \frac{-\overline{z}}{1-|z|^2} - \partial u(z) \right|
\leq \frac{2}{\min\{ \rho(z,f_1/f_2), \rho_a(z,f_1/f_2) \}}, \quad z\in\D.
\end{equation*}
We deduce
\begin{equation*}
|  \nabla u(z) | \leq \frac{2}{1-|z|^2} \left( 1 + \frac{2}{\min\big\{ \rho(z,f_1/f_2), \rho_a(z,f_1/f_2) \big\}} \right),
\quad z\in\D.
\end{equation*}
Denote $h=f_1/f_2$.
It suffices to show that both $\rho(z,h)$ and $\rho_a(z,h)$
are bounded from below by a constant multiple of $\omega(z)/(1-|z|^2)$ as $|z|\to 1^-$.

Let $\psi: \D \to (0,\infty)$ be the weight $\psi(z)=c \, \omega(z)/(1-|z|^2)$, where $0<c<1$ is a~sufficiently
small constant whose value is determined later. By the assumption, we may assume that $\psi: \D \to (0,1/2)$
and therefore $\varphi_a(\psi(a)z) \in \Delta_p(a,1/2)$ for all $a,z\in\D$. By \eqref{eq:regular} and standard estimates,
\begin{equation*}
\sup_{a,z\in\D} \, \frac{\psi(a)}{\psi\big(\varphi_a(\psi(a)z)\big)}
= \sup_{a,z\in\D} \, \frac{\omega(a)}{\omega\big(\varphi_a(\psi(a)z)\big)} \cdot \frac{1-|\varphi_a(\psi(a)z)|^2}{1-|a|^2}
<\infty.
\end{equation*}
Function $h$ is locally univalent and meromorphic, and its Schwarzian derivative
satisfies $S_h= 2A$. Let $g_a(z) = (h \circ \varphi_a)(\psi(a) z)$ for $a,z\in\D$. By the chain rule,
\begin{align*}
|S_{g_a}(z)| & = \big| S_h\big( \varphi_a(\psi(a)z) \big) \big| \, \big| \varphi_a'\big( \psi(a) z\big)\big|^2 \, \psi(a)^2\\
      & \leq \frac{2 \, \nm{A}_{H^\infty_{\omega^2}}}{\omega\big( \varphi_a(\psi(a)z) \big)^2} 
        \, \frac{\big(1-| \varphi_a( \psi(a) z )|^2 \big)^2}{\big(1-|\psi(a)z|^2\big)^2}\, \frac{c^2 \omega(a)^2}{(1-|a|^2)^2},
        \quad a,z\in\D.
\end{align*}
We deduce that $\nm{S_{g_a}}_{H^\infty} \leq \pi^2/2$ for any $a\in\D$, provided
that $0<c<1$ is sufficiently small. Therefore 
$g_a$ is univalent in the unit disc \cite[Theorem~II]{N:1949} for any $a\in\D$.
This is equivalent to the fact that $h$ is univalent in $\Delta_p(a,\psi(a))$
for any $a\in\D$, and therefore $\rho(a,h)\geq \psi(a)$ for  $a\in\D$. 

It remains to estimate $\rho_a(z,h)$. Let $\sigma$ denote the spherical distance on the Riemann sphere.
By the assumption $h^\# \in L^\infty_\omega$, we obtain
\begin{align*}
\sigma\big( h(z), h(\zeta) \big)
        \leq \int_{h(\langle z,\zeta \rangle)} \frac{|d\xi|}{1+|\xi|^2}
       = \int_{\langle z,\zeta \rangle} h^{\#}(\xi) \, |d\xi| 
    \leq \left( \, \sup_{\xi\in \langle z,\zeta \rangle} \frac{1-|\xi|^2}{\omega(\xi)} \right) \, \varrho_h(z,\zeta)
\end{align*}
for any $z,\zeta\in\D$.
If $\zeta\in \Delta_p(z,\psi(z))$, which is a~subset of $\Delta_p(z,1/2)$, then
\begin{equation*}
\sigma\big( h(z), h(\zeta) \big) 
   \leq \left( \, \sup_{\xi\in \langle z,\zeta \rangle} \frac{1-|\xi|^2}{\omega(\xi)} \right)  2 \, \varrho_p(z,\zeta)
 \lesssim \frac{1-|z|^2}{\omega(z)} \cdot \frac{c \, \omega(z)}{1-|z|^2}
\end{equation*}
with an~absolute comparison constant.
Then, $h(z)$ and $h(\zeta)$ cannot be antipodal points if $0<c<1$ is sufficiently small.
Therefore $\rho_a(z,h)\geq \psi(z)$ for $z\in\D$,
which completes the proof of Theorem~\ref{theorem:c2}.
\end{proof}

Corollary~\ref{cor:partialconverse} allows us to reach the desired conclusion $A\in H^\infty_2$
under the assumption $|\nabla u| \in L^\infty_1$. The following lemma shows that,
in this sense, Corollary~\ref{cor:partialconverse} improves \cite[Theorem~7]{G:2018},
according to which the same conclusion holds if the linearly independent solutions $f_1,f_2\in \mathcal{B}$ 
satisfy $\inf_{z\in\D}( |f_1(z)| + |f_2(z)| )>0$.


\begin{lemma} \label{lemma:imp}
The following assertions hold.

\begin{enumerate}
\item[\rm (i)]
If $f_1,f_2\in \mathcal{B}$ are linearly independent solutions 
of \eqref{eq:de2} for $A\in\H(\D)$, and $\inf_{z\in\D}( |f_1(z)| + |f_2(z)| )>0$, 
then $|\nabla u| \in L^\infty_1$ for $u=-\log \, (f_1/f_2)^\#$.

\item[\rm (ii)]
There exists $A\in H^\infty_2$ such that \eqref{eq:de2} admits
linearly independent solutions $f_1,f_2$ such that
$\inf_{z\in\D}( |f_1(z)| + |f_2(z)| )=0$ 
but $|\nabla u| \in L^\infty_1$.

\item[\rm (iii)]
There exists $A\in \H(\D)$ such that \eqref{eq:de2} admits
linearly independent solutions $f_1,f_2$ with
$f_1/f_2$ bounded (and hence normal) but $|\nabla u| \not\in L^\infty_1$.

\end{enumerate}
\end{lemma}


\begin{proof}
(i) Since $f_1,f_2\in\mathcal{B}$ satisfy $\inf_{z\in\D}( |f_1(z)| + |f_2(z)| )>0$, we deduce
\begin{equation*}
\begin{split}
  |\nabla u(z)| & = 2 \left| \partial u(z) \right| 
   = 2 \, \frac{\big| f_1'(z)\overline{f_1(z)}+f_2'(z) \overline{f_2(z)} \big|}{|f_1(z)|^2+|f_2(z)|^2} \\
  & \leq \frac{2 \max\{ \nm{f_1}_{\mathcal{B}},  \nm{f_2}_{\mathcal{B}}\}}{1-|z|^2} \, \frac{|f_1(z)|+|f_2(z)|}{|f_1(z)|^2+|f_2(z)|^2}
\lesssim \frac{1}{1-|z|^2}, \quad z\in\D.
\end{split}
\end{equation*}

(ii) Consider the analytic and univalent function $h(z)=-\log(1-z)$, $z\in\D$.
Define $A=S_h/2$, where $S_h$ is the Schwarzian derivative of $h$.
Then, $A(z)=4^{-1}(1-z)^{-2}$, $z\in\D$, and clearly $A\in H^\infty_2$.
It is well-known that \eqref{eq:de2}
admits two linearly independent solutions $f_1,f_2$ such that $h=f_1/f_2$.
In this case
\begin{equation*}
 \frac{|W(f_1,f_2)|}{|f_1(z)|^2+|f_2(z)|^2} = h^{\#}(z) = \frac{1}{|1-z| \big( 1 + | \log(1-z)|^2 \big)}, \quad z\in\D,
\end{equation*}
is unbounded in $\D$, while $|\nabla u| \in L^\infty_1$ by Corollary~\ref{cor:partialconverse} ($h$ is normal as it is univalent).

Part (iii) follows by the proof of Theorem~\ref{thm:converse}(ii). An~application of
Corollary~\ref{cor:partialconverse} reveals that $|\nabla u| \not\in L^\infty_1$.
\end{proof}

It is a~natural question to ask how $|\nabla u |\in L^\infty_1$ compares
to Theorem~\ref{thm:converse_final}? On one hand,
Lemma~\ref{lemma:imp}(ii) serves as an~example where $|\nabla u| \in L^\infty_1$ but
\eqref{eq:exc_ind} fails for any pairwise disjoint pseudo-hyperbolic discs (consider the positive real axis). 
On the other hand,
Example~\ref{ex:cases}(ii) in Section~\ref{sec:6} implies that there are cases in which
\eqref{eq:exc_ind} is satisfied but $|\nabla u |\not\in L^\infty_1$ ($f_1/f_2$ is not normal). In both of these
examples, the coefficient function satisfies $A\in H^\infty_2$.


\begin{proof}[Proof of Corollary~\ref{cor:c2}]
The assertions (i) and (ii) are equivalent by Theorem~\ref{theorem:c2}. Note that
(i) implies (iii) by Lemma~\ref{lemma:derprop}, while (iii) implies (i), and also (ii), by Theorem~\ref{thm:repres}(ii).
Finally, (ii) is equivalent to (iv) according to Theorem~\ref{thm:repres}(i).
\end{proof}

The arguments in this section are build on the representation \eqref{eq:preli} for the coefficient $A$.
Derivatives of the coefficient can be controlled by expressions of similar type. For example,
by differentiating \eqref{eq:de2} we obtain $f'''+A'f +Af'=0$, and 
\begin{align*}
|A| & = \frac{|f_1'|+|f_2'|}{|f_1'|+|f_2'|} \, |A| = \frac{|f_1'''+A'f_1| + |f_2'''+A'f_2|}{|f_1'|+|f_2'|}
\geq |A'| \, \frac{|f_1|+|f_2|}{|f_1'|+|f_2'|} - \frac{|f_1'''|+|f_2'''|}{|f_1'|+|f_2'|}.
\end{align*}
Therefore, by applying \eqref{eq:preli},
\begin{equation*}
|A'| \leq \frac{|f_1'|+|f_2'|}{|f_1|+|f_2|} \left( |A| + \frac{|f_1'''|+|f_2'''|}{|f_1'|+|f_2'|} \right)
 = \frac{|f_1'|+|f_2'|}{|f_1|+|f_2|} \cdot \frac{|f_1''|+|f_2''|}{|f_1|+|f_2|} +  \frac{|f_1'''|+|f_2'''|}{|f_1|+|f_2|}. 
\end{equation*}


\section{Proof of Theorem~\ref{thm:fixedpoints}}

It is natural to require that solution with prescribed fixed points is bounded in~$\D$.
Under this requirement, Theorem~\ref{thm:fixedpoints} is best possible. This is a~consequence of
properties, which are restated in Lemma~\ref{lemma:oo} for convenience.


\begin{proof}[Proof of Theorem~\ref{thm:fixedpoints}]
Let $B=B_{\{z_n\}}$ be the Blaschke product with zeros $\{z_n\}$.
Let $0<\varepsilon<1$, and define $f_1(z) = z + \varepsilon z^3 B(z)$, $z\in\D$.
The fixed points of $f_1$ are precisely $\{0\} \cup \{z_n\}$.
By the Schwarz lemma $| z^3\,  B(z) | \leq |z|$ for $z\in\D$, 
and therefore $(1-\varepsilon) |z|\leq |f_1(z)| \leq (1+\varepsilon)|z|$ for any $z\in\D$.

Since $f_1$ has only one zero in $\D$ and $f_1''(0)=0$, 
we deduce $A=-f_1''/f_1 \in \H(\D)$. If $0<\delta<1$, then
\begin{align*}
\sup_{\delta<|z|<1} |A(z)| 
\leq \frac{\varepsilon}{(1-\varepsilon)\delta} \, \sup_{\delta<|z|<1} \Big( |B''(z)| + 6 \, |B'(z)| + 6 \, |B(z)| \Big),
\end{align*}
and consequently, $|A(z)|^2(1-|z|^2)^3\, dm(z)$ is a~Carleson measure.
If $f_2$ is defined by \eqref{eq:solrep} for fixed $\alpha\in\D\setminus\{0\}$, then $f_2\in H^\infty$ 
is a~solution of \eqref{eq:de2} and is linearly independent to $f_1$.
Consequently, all solutions of \eqref{eq:de2} are bounded.
\end{proof}


\begin{lemma} \label{lemma:oo}
The following assertions hold.
\begin{enumerate}
\item[\rm (i)]
The identity function is the only one
in $\{ f\in H^\infty : \nm{f}_{H^\infty} \leq 1\}$ which has more
than one fixed point. 

\item[\rm (ii)]
The identity function is the only
one in $\NC$ which has
more fixed points than the Blaschke condition allows.
\end{enumerate}
\end{lemma}


\begin{proof}
(i) Suppose that $f\in H^\infty$, $\nm{f}_{H^\infty} \leq 1$ and $f$ has two distinct fixed points 
$\alpha,\beta\in\D$. If one of the fixed points is zero, then $f(z)\equiv z$ by
the classical Schwarz lemma \cite[Lemma~1.1, p.~1]{G:2007}. Otherwise, use the property
$\varrho_p(f(\alpha),f(\beta)) = \varrho_p(\alpha,\beta)$ and 
the Schwarz-Pick lemma \cite[Lemma~1.2, p.~2]{G:2007} to conclude that
$f$ is a~M\"obius transformation. Define $g = \varphi_\alpha \circ f \circ \varphi_\alpha$,
and note that $g$ is a~M\"obius transformation which fixes the values 
$0$ and $\varphi_\alpha(\beta)\neq 0$, and hence $g(z)\equiv z$ by the classical Schwarz lemma.
Therefore $f$ is the identity function.

(ii) Suppose that $f$ is bounded in $\D$ and 
its fixed points do not satisfy the Blaschke condition. Now $f(z)-z$ belongs to~$\NC$ and
has more zeros than the Blaschke condition allows. The claim follows.
\end{proof}


\section{Proofs of Theorems~\ref{thm:fixedpoints2} and \ref{thm:fixedpoints2N}}


\begin{proof}[Proof of Theorem~\ref{thm:fixedpoints2}]
Let $\Lambda \subset \D\setminus\{0\}$ be a~uniformly separated sequence.
Then, the corresponding Blaschke product $B= B_\Lambda$ satisfies
\eqref{eq:sep}. 

Let $h\in H^\infty$ be a~function which satisfies $h(z_n)=\log{z_n}$ for $z_n\in\Lambda$.
The existence of such $h$ is guaranteed by Carleson's interpolation theorem
\cite[Theorem~3]{C:1962}. Let $\{C_n\}$ be the sequence of real numbers defined as follows:
Whenever $z_n\in\Lambda$ is prescribed to be an~attractive fixed point define $C_n=1/2$,
if neutral choose $C_n=1$, while otherwise take $C_n=2$.
By \eqref{eq:sep}, we obtain
\begin{equation*}
\sup_{z_n\in\Lambda} \, \left| \frac{1}{B'(z_n)} \left( \frac{C_n}{z_n} - h'(z_n) \right) \right|
\leq \sup_{z_n\in\Lambda} \, \frac{1-|z_n|^2}{\delta} \left( \frac{2}{\inf_n |z_n|} + |h'(z_n)| \right) < \infty,
\end{equation*}
and hence $\{w_n\}= \{(C_n/z_n - h'(z_n))/B'(z_n)\}$ is a bounded sequence. 
The aforementioned Carleson's result guarantees that there exists $g\in H^\infty$
with $g(z_n)=w_n$ for $z_n\in\Lambda$.
Define $f_1=\exp( h+Bg)$, and note that $f_1$ is not only in $H^\infty$ but also is uniformly 
bounded away from zero. Moreover, 
\begin{equation*}
f_1(z_n)=z_n, \quad f_1'(z_n) = z_n \Big( h'(z_n)+B'(z_n) g(z_n) \Big) = C_n, \quad z_n\in\Lambda.
\end{equation*}
The points $z_n\in\Lambda$ are fixed points of the prescribed type. 
The coefficient $A= -f_1''/f_1 \in\H(\D)$ satisfies $|A|\lesssim |f_1''|$ in $\D$, and therefore
$|A(z)|^2(1-|z|^2)^3\, dm(z)$ is a~Carleson measure.
The fact that all solutions of \eqref{eq:de2} are bounded follows as in the proof of Theorem~\ref{thm:fixedpoints}.
\end{proof}

Note that the solution $f_1$ in Theorem~\ref{thm:fixedpoints2}, which has prescribed fixed points
of pregiven type,  may have fixed points which do not belong to $\Lambda$.


\begin{remark}
If $A\in\H(\D)$ and $z_0\in\D$, then \eqref{eq:de2} admits a~unique
solution~$f$ such that the initial conditions $f(z_0)=\alpha\in\C$ and $f'(z_0)=\beta\in\C$ are satisfied.
Therefore  fixed points of solutions of \eqref{eq:de2} are not always distinct from zeros or critical points.
In the proof of Theorem~\ref{thm:fixedpoints2}, $\{C_n\}\subset\C$ can be any sequence 
with the property $\sup_n |C_n|(1-|z_n|^2)<\infty$. If we take $C_n=0$ for all 
$n$, then every point $z_n\in\Lambda$ is not only a~fixed point but also a~critical point of the solution $f_1$.
\end{remark}


\begin{proof}[Proof of Theorem~\ref{thm:fixedpoints2N}]
Let $\Lambda\in\Int\NC$ be the sequence of non-zero points, and let $B=B_\Lambda$
be the corresponding Blaschke product. Since $\Lambda\in\Int\NC$, \cite[Theorem~1.2]{HMNT:2004} implies that
there exists $h_1\in\hp(\D)$ such that \eqref{eq:nevii} holds.

Let $h\in \NC$ be a~function which satisfies $h(z_n)=\log{z_n}$ for $z_n\in\Lambda$.
Since $\Lambda$ is Nevanlinna interpolating, 
the existence of such function $h$ is guaranteed by 
\cite[Theorem~1.2]{HMNT:2004}. Let $\{C_n\}$ be the sequence of real numbers defined 
as in the proof of Theorem~\ref{thm:fixedpoints2}.
As $h\in\NC$, \eqref{eq:nevder} implies that there exists a~constant $0<C<\infty$ and $h_2\in\hp(\D)$ such that
\begin{equation*}
\left| \frac{1}{B'(z_n)} \left( \frac{C_n}{z_n} - h'(z_n) \right) \right|
\leq \frac{1-|z_n|^2}{e^{-h_1(z_n)}} \left( \frac{2}{\inf_n |z_n|} + \frac{e^{C}e^{h_2(z_n)}}{1-|z_n|^2} \right), \quad z_n\in\Lambda.
\end{equation*}
Since $\{w_n\}= \{(C_n/z_n - h'(z_n))/B'(z_n)\} \in \NC \, | \, \Lambda$ by
\cite[Theorem~1.2]{HMNT:2004}, there exists $g\in \NC$
with $g(z_n)=w_n$ for $z_n\in\Lambda$. Define $f=\exp( h+Bg)$, and note that
\begin{equation*}
f(z_n)=z_n, \quad f'(z_n) = z_n \Big( h'(z_n)+B'(z_n) g(z_n) \Big) = C_n, \quad z_n\in\Lambda.
\end{equation*}
The points $z_n\in\Lambda$ are fixed points of the prescribed type. 
Finally, the coefficient $$A= -f''/f = - \big( (h+Bg)' \big)^2 - (h+Bg)''\in\H(\D)$$ satisfies
$|A(z)|(1-|z|^2)^2 \leq e^{H(z)}$, $z\in\D$ and $H\in \hp(\D)$,
by~\eqref{eq:nevder}.
\end{proof}

\smallskip


\section*{Acknowledgements} 

The author thanks Artur Nicolau for helpful conversations and guidance, especially
in relation to Nevanlinna interpolating sequences.
The author gratefully acknowledges the hospitality of Departament de Matem\`atiques, 
Universitat Aut\`onoma de Barcelona.


\end{document}